\documentclass[12pt,oneside]{amsart}

\usepackage{amssymb,amsmath}
\usepackage{amsfonts}
\usepackage{xcolor} 
\usepackage{multirow} 
\usepackage{booktabs}
\usepackage{geometry}
\usepackage{graphicx}
\usepackage{float}
\usepackage{amsthm}
\usepackage{soul}
\usepackage[pagewise, mathlines]{lineno}
\usepackage{tikz}

\newcommand*{\DashedArrow}[1][]{\mathbin{\tikz [baseline=-0.25ex,-latex, dashed,#1] \draw [#1] (0pt,0.5ex) -- (1.3em,0.5ex);}}%

\newtheorem{theorem}{Theorem}[section]
\newtheorem{lemma}[theorem]{Lemma}
\newtheorem{proposition}[theorem]{Proposition}

\newtheorem{question}[theorem]{Question}

\theoremstyle{definition}

\newtheorem{remark}[theorem]{Remark}
\newtheorem{example}[theorem]{Example}

\numberwithin{equation}{section}
\newtheorem{innercustomthm}{Theorem}
\newenvironment{customthm}[1]
  {\renewcommand\theinnercustomthm{#1}\innercustomthm}
  {\endinnercustomthm}
\DeclareMathOperator{\rk}{Rk}
\DeclareMathOperator{\Cr}{cr}
\DeclareMathOperator{\brk}{brk}
\DeclareMathOperator{\nullity}{nullity}
\DeclareMathOperator{\HF}{HF}

\everymath{\displaystyle}

\title[Waring decompositions of the product of two quadrics]{Waring decompositions of the product of two quadrics: the small rank cases}

\author[Bhat]{Meghana Bhat}
\address[Meghana Bhat]{Department of Mathematics, Indian Institute of Technology Dharwad, Permanent Campus, Chikkamalligwad,  Dharwad - 580011, Karnataka, India}
\email{mbhat.math@gmail.com}

\author[Carlini]{Enrico Carlini}
\address[Enrico Carlini]{DISMA Department of Mathematical Sciences, Politecnico Di Torino,  Corso Duca degli Abruzzi, 24, Torino - 10129, Italy}
\email{enrico.carlini@polito.it}

\author[Dubey]{Saipriya Dubey}
\address[Saipriya Dubey]{Chennai Mathematical Institute, Plot H1 SIPCOT IT Park, Siruseri, Kelambakkam - 603103, Tamil
Nadu, India}
\email{saipriya721@gmail.com}

\author[Masuti]{Shreedevi K. Masuti}
\address[Shreedevi K. Masuti]{Department of Mathematics, Indian Institute of Technology Dharwad, Permanent Campus, Chikkamalligwad,  Dharwad - 580011, Karnataka, India}
\email{shreedevi@iitdh.ac.in}

\thanks{MB is supported by Prime Minister's Research Fellowship (PMRF), Govt. of India. EC acknowledge that this study was carried out within the PRIN 2022 PNRR project grant P2022E2Z4AK “0-Dimensional Schemes, Tensor Theory, and Applications” funded by European Union– Next Generation EU within the PRIN 2022 program (D.D. 104- 02/02/2022 Ministero dell’Università e della Ricerca).
SD is partially supported by a grant from the Infosys Foundation. SKM is supported by
CRG grant CRG/2022/007572 and MATRICS grant MTR/2022/000816  funded by SERB,  Govt. of India. The project is partially supported by SPARC grant SPARC/2019-2020/P1566/SL}

\date{\today}
\begin{document}

\begin{abstract}
    {In this paper we study forms of the type $(x_1^2+ \cdots +x_m^2)(y_1^2+ \cdots+y_n^2)$ using projections. For $m=1, m=2$, and for any $n$ we describe: the geometry, the forbidden locus, the Hilbert function, and the ideal of all minimal apolar sets. {Further, we compute the cactus rank, a bound on the border rank, and the dimension of the Variety of Sums of Powers.} For $m,n \geq 3,$ we provide new lower and upper bounds for the Waring rank. }
\end{abstract}
\maketitle


\section{Introduction}
Let $S=\mathbb{C}[x_0, \ldots, x_n]$ be the standard graded polynomial ring, and let  $F \in S$ be a homogeneous polynomial of degree $d$, also known as a degree $d$ form. It is well known that there exist $L_1, \ldots, L_r \in S_1$ such that \[F=\sum_{i=1}^{r}\alpha_iL_i^d,\]  for some $\alpha_i \in \mathbb{C},$ $1 \leq i \leq r.$
Such a decomposition of $F$ is called {\it Waring decomposition} or {\it sum of powers decomposition}.  The least integer $r$ such that there exists a Waring decomposition of $F$ with exactly $r$ summand is  called {\it Waring rank} of $F,$ denoted by $\rk F.$
\begin{question}\label{Question:ComputeWaringRank}
    Given $F\in S_d,$ compute the Waring rank and a Waring decomposition of $F$.
\end{question}
 This problem is a specific instance of the broader problem of tensor decomposition. Since homogeneous polynomials can be identified with symmetric tensors, Waring decompositions of forms are  equivalent to decompositions of symmetric tensors into rank one {symmetric} terms. Consequently, Question \ref{Question:ComputeWaringRank} is highly relevant to fields relying on efficient tensor decompositions, such as Algebraic Statistics, Signal Processing \cite{Lan12}, Quantum entanglement \cite{BC12, QuantumInformationTheory}, and Complexity Theory \cite{NPhard-TensorRank, Shitov16}, {for example in the study of complexity of matrix multiplications \cite{polynomialMatrixMultiplication18,Lan14GCT,Lan22,LO15}, and more generally in circuit complexity \cite{DGIJL26, Kumar20, Survey}.}
One of the most celebrated result in the Waring rank  problems is due to J. Alexander and A. Hirschowitz where they  determined the rank of a generic form \cite{AH95}. 
{However, despite our understanding of the generic case, computing the Waring rank of a specific polynomial remains a challenging problem. In fact, in \cite{Shitov16} it is proved that computing the Waring rank is an NP-hard problem.}

Significant progress has been made for specific families of polynomials. For example, the Waring rank of binary forms was classically known due to Sylvester. In \cite{sylvester1851lx,Sylvester1851} an  algorithm to compute the Waring rank of binary forms is presented. Similarly, the Waring rank of a quadratic form in any number of variables is classically known, as it coincides with the rank of the associated symmetric matrix.

In more recent years, the rank has been determined for several other families:
the Waring rank of monomials are obtained in \cite{BBT13, CCG12}. The waring rank of binary binomials, and of elementary symmetric polynomials are computed in \cite{BM21} and \cite{Lee16}. 
The study of reducible forms has also been a subject of long-standing interest due to their importance in analysis. While the rank of a quadratic form itself is elementary, the ranks of its powers are far more intricate and have been studied extensively in \cite{Flavi25, Rez92}.
The Waring rank of reducible cubic forms was specifically investigated in \cite{CVG16}. More recently, the rank of broader classes of reducible forms has been explored in \cite{GG24}.

{
Given a homogeneous form $F \in S_d$, even when the $\rk F$ is known, it does not give any information on which linear forms  appear in the Waring decomposition. For $L = \sum_{i=0}^na_ix_i \in S_1$ we associate a point $P=[a_0: \cdots:a_n] \in \mathbb{P}^n$. Given a Waring decomposition  $F= \sum_{i=1}^rL_i^d$, hence we associate a set of points $\mathbb{X}=\{P_1, \ldots, P_r\}$ to $F$ where each $P_i$ corresponds to $L_i$ for $1 \leq i \leq r.$ In this case we call $\mathbb{X}$, a set of apolar points to $F$. Further, $\mathbb{X}$ is said to be a minimal apolar set if $|\mathbb{X}| = \rk F$.
\begin{question}\label{Question:Structure}
    Given $F\in S_d,$ describe algebraic and geometric properties of the {minimal apolar sets}  of $F$.
\end{question}
In \cite{BBT13}, the authors proved that any minimal apolar scheme to a monomial is a complete intersection. A systematic investigation of these sets was initiated in \cite{ForbiddenLocus}, where the authors introduced the concepts of Waring loci and forbidden loci. In this paper, we continue that line of inquiry by focusing on the product of two quadratic forms. Note that in {\cite[Proposition 4.4]{CCCGW18}} it was proved that $\rk  x^2(y_1^2+\ldots+y_n^2)=3n$, while the value of $\rk \, (x_1^2+x_2^2)(y_1^2+\ldots+y_n^2)$ was unknown. Our main results are the following theorems.}

\begin{customthm}{A}\label{thm:A}
    Let $F=x^2(y_1^2+\cdots+y_n^2)$ with $n\geq 2$. If $\mathbb{X}$ is a set of $\rk F=3n$ distinct points apolar to $F$, then the following facts hold:
    \begin{itemize}
        \item[i)] The forbidden locus of $F$ is
        \[
        \mathcal{F}_F=V(Y_1^2+\cdots+Y_n^2).
        \]
        \item[ii)] If $\mathbb{Y}$ is the projection of $\mathbb{X}$ from $V(Y_1,\ldots,Y_n)$ on $V(X)$, then $\mathbb{Y}$ is a set of $n$ distinct points in linear general position, that is
        \[
        \begin{array}{r|c|c|c}
        i & 0 & 1 & 2 \\
        \hline
        \HF(\mathbb{Y},i) & 1 & n & n \\
             
        \end{array},
        \]
        and $\HF(\mathbb{Y},i)=n$ for $i\geq 3$.
        \item[iii)] Let $\mathbb{Y}=\{Q_1,\ldots,Q_n\}$ {and let $P = V(Y_1, \ldots, Y_n) = [1:0: \cdots:0]$}. If $\Lambda_i=\langle  Q_i,P \rangle,$  where $\langle Q_i,P\rangle$ denotes the linear span of $Q_i$, then $\mathbb{X}\cap\Lambda_i$ is a set of three distinct points for $1\leq i\leq n$.
        \item[iv)] The Hilbert function of $\mathbb{X}$ is
        \[
        \begin{array}{r|c|c|c|c}
        i & 0 & 1 & 2 & 3 \\
        \hline
        \HF(\mathbb{X},i) & 1 & n & 2n+1 & 3n \\
             
        \end{array},
        \]
        and $\HF(\mathbb{X},i)=3n$ for $i\geq 3$.
        \item[v)] {If $\mathbb{X}$ is a minimal apolar set to $F$, then there exists $\sigma \in O(n)$, the orthogonal group on the variables $Y_1, \ldots, Y_n$, such that \[
        I(\sigma(\mathbb{X}))=(\{Y_iY_j\}_{1\leq i<j\leq n},\, X^3+ \sum_{i=2}^nL_i(Y_1^2-Y_i^2))
        \]}
        for {some} linear forms $L_i\in \mathbb{C}[X, Y_1, \ldots, Y_n], 2 \leq i \leq n$. 
    \end{itemize}
\end{customthm}
 \begin{figure}[H]
 \centering
    \def\svgwidth{0.5\textwidth}
\begingroup%
  \makeatletter%
  \providecommand\color[2][]{%
    \errmessage{(Inkscape) Color is used for the text in Inkscape, but the package 'color.sty' is not loaded}%
    \renewcommand\color[2][]{}%
  }%
  \providecommand\transparent[1]{%
    \errmessage{(Inkscape) Transparency is used (non-zero) for the text in Inkscape, but the package 'transparent.sty' is not loaded}%
    \renewcommand\transparent[1]{}%
  }%
  \providecommand\rotatebox[2]{#2}%
  \newcommand*\fsize{\dimexpr\f@size pt\relax}%
  \newcommand*\lineheight[1]{\fontsize{\fsize}{#1\fsize}\selectfont}%
  \ifx\svgwidth\undefined%
    \setlength{\unitlength}{432.39949036bp}%
    \ifx\svgscale\undefined%
      \relax%
    \else%
      \setlength{\unitlength}{\unitlength * \real{\svgscale}}%
    \fi%
  \else%
    \setlength{\unitlength}{\svgwidth}%
  \fi%
  \global\let\svgwidth\undefined%
  \global\let\svgscale\undefined%
  \makeatother%
  \begin{picture}(1,0.62680233)%
    \lineheight{1}%
    \setlength\tabcolsep{0pt}%
    \put(0,0){\includegraphics[width=\unitlength,page=1]{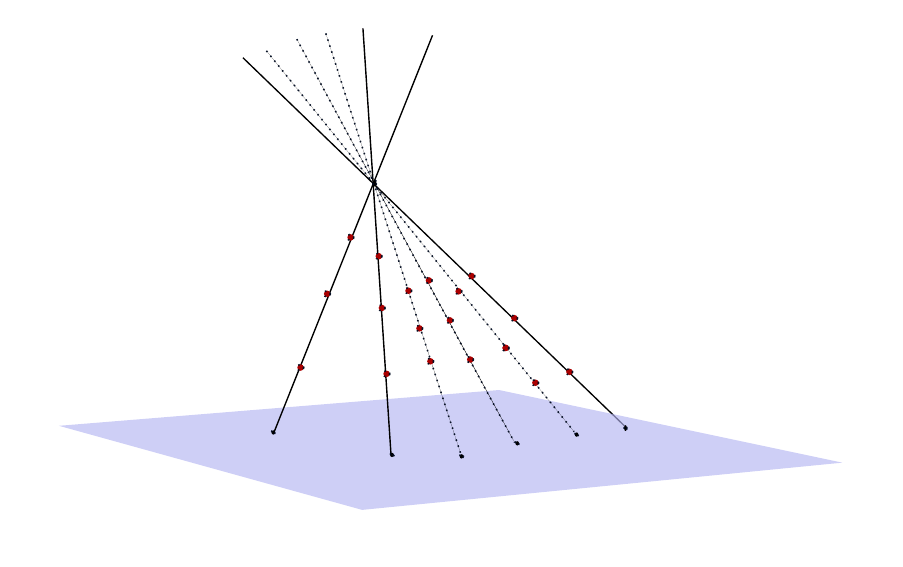}}%
    \put(0.43177932,0.42776651){\color[rgb]{0.10196078,0.10196078,0.10196078}\makebox(0,0)[lt]{\lineheight{1.25}\smash{\begin{tabular}[t]{l}$P$\end{tabular}}}}%
    \put(0.31219731,0.13402815){\color[rgb]{0.10196078,0.10196078,0.10196078}\makebox(0,0)[lt]{\lineheight{1.25}\smash{\begin{tabular}[t]{l}$Q_1$\end{tabular}}}}%
    \put(0.44060753,0.10513586){\color[rgb]{0.10196078,0.10196078,0.10196078}\makebox(0,0)[lt]{\lineheight{1.25}\smash{\begin{tabular}[t]{l}$Q_2$\end{tabular}}}}%
    \put(0.69903305,0.15007944){\color[rgb]{0.10196078,0.10196078,0.10196078}\makebox(0,0)[lt]{\lineheight{1.25}\smash{\begin{tabular}[t]{l}$Q_n$\\\end{tabular}}}}%
    \put(0.86757143,0.08988715){\color[rgb]{0.10196078,0.10196078,0.10196078}\makebox(0,0)[lt]{\lineheight{1.25}\smash{\begin{tabular}[t]{l}V(X)\end{tabular}}}}%
  \end{picture}%
\endgroup%

     \caption{A minimal apolar set of $F=x^2(y_1^2+\cdots+y_n^2)$: red dots represent apolar points contained in $n$-lines with 3 points on each line.}
     \label{Picture:PointsOnLines}
    \end{figure}
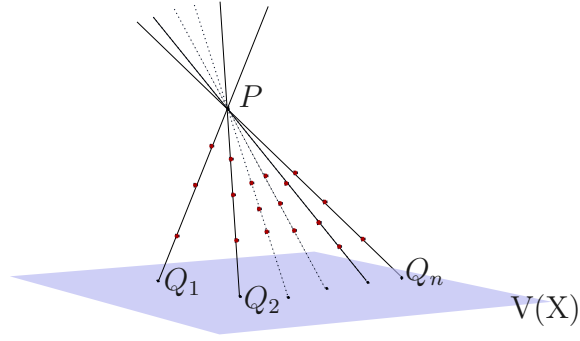
\begin{customthm}{B}\label{thm:B}
    Let $F=(x_1^2+x_2^2)(y_1^2+\cdots+y_n^2)$ and $n\geq 2$, then $\rk F = 4n$. If $\mathbb{X}$ is a set of $4n$ distinct points apolar to $F$, then the following facts hold:
    \begin{itemize}
        \item[i)] The forbidden locus of $F$ is given by $$\mathcal{F}_F=V(X_1^2+X_2^2)\cup V(Y_1^2+\cdots+Y_n^2).$$
        \item[ii)] If $\mathbb{Y}$ is the projection of $\mathbb{X}$ from $V(Y_1,\ldots,Y_n)$ on $V(X_1,X_2)$, then $\mathbb{Y}$ is a set of $n$ distinct points in linear general position, that is
        \[
        \begin{array}{r|c|c|c}
        i & 0 & 1 & 2 \\
        \hline
        \HF(\mathbb{Y},i) & 1 & n & n \\
             
        \end{array},
        \]
        and $\HF(\mathbb{Y},i)=n$ for $i\geq 3$.
         
        \item[iii)] Let $\mathbb{Y}=\{Q_1,\ldots,Q_n\}$ and let $\ell=V(Y_1,\ldots,Y_n)$. If $\Lambda_i=\langle \ell, Q_i \rangle$, where $\langle \ell, Q_i \rangle$ denotes the linear span of $\ell$ and $Q_i$, then $\mathbb{X}\cap\Lambda_i$ is a set of four distinct points for $1\leq i\leq n$ contained in two lines passing through $Q_i$.
       
        \item[iv)]  {If $\mathbb{X}$ is a minimal apolar set to $F$, then there exists $\sigma \in O(n)$, the orthogonal group on the variables $Y_1, \ldots, Y_n$, such that \[
        I(\sigma(\mathbb{X}))=(\{Y_iY_j\}_{1\leq i<j\leq n},\mathcal{Q}_1,\mathcal{Q}_2)
        \]}
        where $\mathcal{Q}_1=X_1^2+\sum_{i=2}^n \alpha_i(Y_1^2-Y_i^2)$, $\mathcal{Q}_2=X_2^2+\sum_{i=2}^n \beta_i(Y_1^2-Y_i^2)$ for some $\alpha_i, \beta_i \in \mathbb{C}\setminus\{0\},$ $ 2 \leq i \leq n$ such that $\sum_{i=2}^n\alpha_i \neq 0$, $\sum_{i=2}^n \beta_i \neq 0.$
        \item[v)] If $\mathbb{W}$ is the projection of $\mathbb{X}$ from $V(X_1,X_2)$ on $V(Y_1,\ldots,Y_n)$, then $\mathbb{W}$ is a set of $2k$ points with $k\leq n$.
    \end{itemize}
\end{customthm}
\begin{figure}[H]
\centering
    \def\svgwidth{0.5\textwidth}
\begingroup%
  \makeatletter%
  \providecommand\color[2][]{%
    \errmessage{(Inkscape) Color is used for the text in Inkscape, but the package 'color.sty' is not loaded}%
    \renewcommand\color[2][]{}%
  }%
  \providecommand\transparent[1]{%
    \errmessage{(Inkscape) Transparency is used (non-zero) for the text in Inkscape, but the package 'transparent.sty' is not loaded}%
    \renewcommand\transparent[1]{}%
  }%
  \providecommand\rotatebox[2]{#2}%
  \newcommand*\fsize{\dimexpr\f@size pt\relax}%
  \newcommand*\lineheight[1]{\fontsize{\fsize}{#1\fsize}\selectfont}%
  \ifx\svgwidth\undefined%
    \setlength{\unitlength}{544.99644362bp}%
    \ifx\svgscale\undefined%
      \relax%
    \else%
      \setlength{\unitlength}{\unitlength * \real{\svgscale}}%
    \fi%
  \else%
    \setlength{\unitlength}{\svgwidth}%
  \fi%
  \global\let\svgwidth\undefined%
  \global\let\svgscale\undefined%
  \makeatother%
  \begin{picture}(1,0.86764281)%
    \lineheight{1}%
    \setlength\tabcolsep{0pt}%
    \put(0,0){\includegraphics[width=\unitlength,page=1]{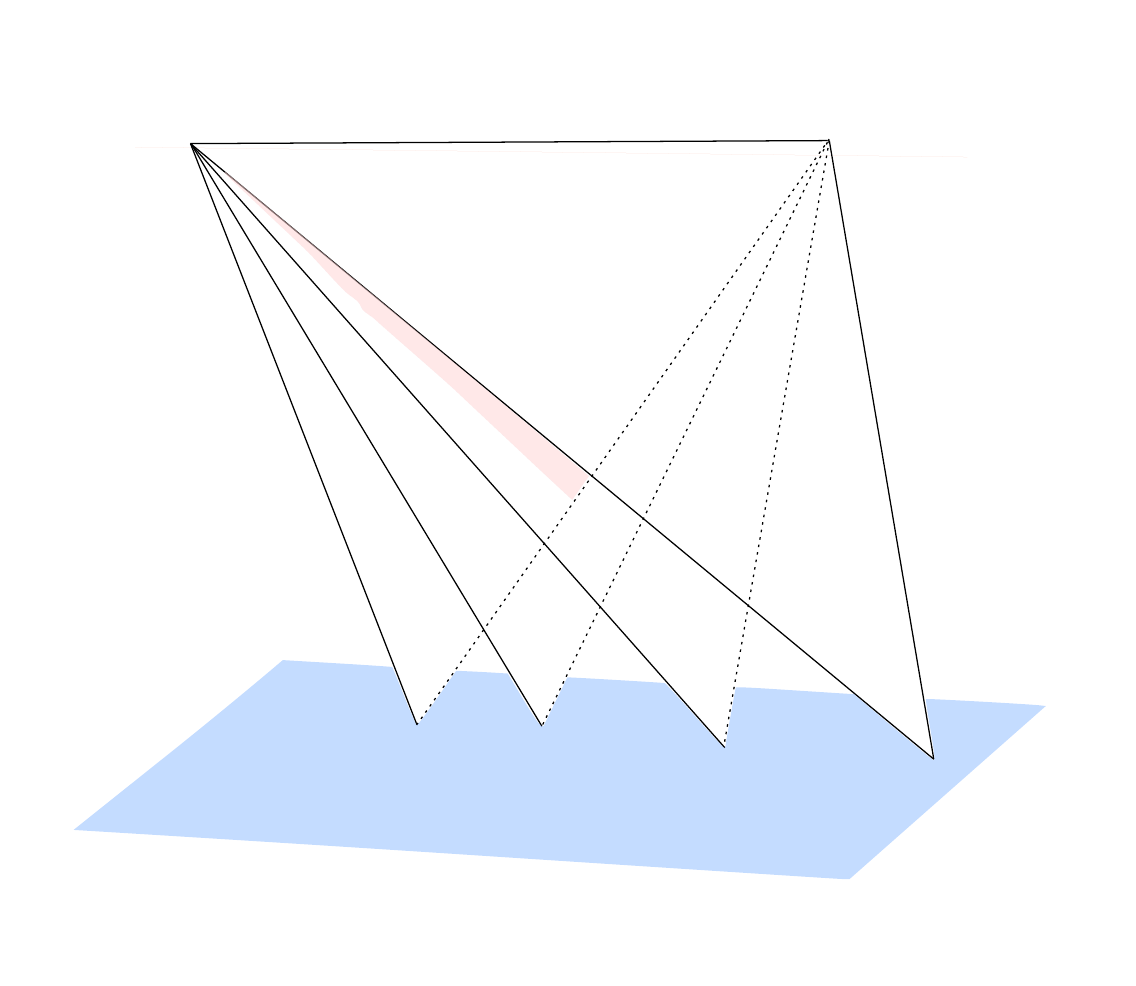}}%
    \put(0.73404207,0.10329408){\color[rgb]{0,0,0}\makebox(0,0)[lt]{\lineheight{1.25}\smash{\begin{tabular}[t]{l}        $V(X_1, X_2)$\end{tabular}}}}%
    \put(0.82630474,0.19420771){\color[rgb]{0,0,0}\makebox(0,0)[lt]{\lineheight{1.25}\smash{\begin{tabular}[t]{l}$Q_n$\end{tabular}}}}%
    \put(0.36634867,0.21037025){\color[rgb]{0,0,0}\makebox(0,0)[lt]{\lineheight{1.25}\smash{\begin{tabular}[t]{l}$Q_1$\end{tabular}}}}%
    \put(0.48150607,0.21239047){\color[rgb]{0,0,0}\makebox(0,0)[lt]{\lineheight{1.25}\smash{\begin{tabular}[t]{l}$Q_2$\end{tabular}}}}%
    \put(0.746166,0.7444042){\color[rgb]{0,0,0}\makebox(0,0)[lt]{\lineheight{1.25}\smash{\begin{tabular}[t]{l}$V(Y_1,Y_2, \ldots , Y_n)$\end{tabular}}}}%
    \put(0,0){\includegraphics[width=\unitlength,page=2]{bookpagesLike5.pdf}}%
  \end{picture}%
\endgroup%

     \caption{A minimal apolar set of $F=(x_1^2+x_2^2)(y_1^2+\cdots+y_n^2)$: black dots represent minimal apolar points contained in $n$-planes with 4 points on each plane.}
     \label{Picture:PointsOnPlane}
    \end{figure}

{We note that our projection approach is not going to work for forms that split as a product in higher degrees. The following example shows minimal apolar subsets that are not projectively equivalent.
\begin{example}
    Let $F=x(y_1^3+y_2^3)$. By \cite[Proposition 4.7]{CCCGW18} we have $\rk F = 6.$ Consider   the apolar sets $\mathbb{X}, \mathbb{X}'$, given  by 
    \begin{equation*}
        I({\mathbb{X}}) = (y_1y_2, x^3+y_1^3-y_2^3), \, I({\mathbb{X}'}) = (x^2+y_1y_2, x^3+y_1^3-y_2^3). 
    \end{equation*}
    {Note that the points of $\mathbb{X}$ lie on a unique conic: the union of two lines. Whereas the points of $\mathbb{X}'$  lie on a  unique conic which is irreducible. Hence, no linear change of coordinate can take $I(\mathbb{X})$ to $I(\mathbb{X}').$}
\end{example}
}

Our paper is organized as follows. In section \ref{Sec:Background} we recall some well-known results which are useful to obtain bounds for the Waring rank. In Section \ref{Sec:Basic results} we prove some technical lemmas on upper bounds for Hilbert function and change of coordinates. In sections 4, and 5 we study the specific forms $x^2(y_1^2+\cdots+y_n^2)$ and $(x_1^2+x_2^2)(y_1^2+\cdots+y_n^2)$. In particular, we study the geometry, the defining ideal, and the Hilbert function  of all minimal apolar sets. The results are summarised in Theorem \ref{thm:A} and Theorem \ref{thm:B}. In section 6, we present new lower and upper bounds for the Waring rank of forms of the type  $(x_1^2+ \cdots +x_m^2)(y_1^2+ \cdots+y_n^2).$ {In the final section, we present some results computing the cactus rank, border rank, and the dimension of VSP of the forms considered in Theorem A and Theorem B.}
\vskip 2mm
\noindent {\bf Acknowledgements}: The first, third and fourth authors would like to thank the second author and the Politecnico di Torino for the hospitality, where most of this work was done during their visit. The second author would like to thank IIT Dharwad for the kind hospitality during the phases of this research. We thank the anonymous referee for their suggestions and comments that have considerably improved the presentation of this paper. In particular, the whole last section was developed based on their suggestions.

\section{Background material}\label{Sec:Background}

    An important tool, which is useful in the study of Waring rank problems is the apolarity action. 
    Let $T=\mathbb{C}[X_0, \ldots, X_n]$ be a polynomial ring over $\mathbb{C}.$ The apolarity action of $T$ on $S$ is defined on monomials by
    \[X_0^{a_0} \cdots X_n^{a_n} \circ (x_0^{b_0}\cdots x_n^{b_n}) = \frac{\partial^{a_0}}{\partial x_0^{a_0}}\cdots \frac{\partial^{a_n}}{\partial x_n^{a_n}}(x_0^{b_0}\cdots x_n^{b_n}),\]
    and then extended linearly on $S.$ 
    Given a homogeneous form $F \in S_d,$ using the apolarity action, we write annihilator ideal of $F$ as 
    $F^{\perp} = \{ g \in T \mid g \circ F = 0\}.$ Note that $F^{\perp}$ is an Artinian Gorenstein ideal. 
    Consider a Waring decomposition $F = \sum_{i=1}^{r}\alpha_iL_i^d,$ with $L_i = \sum_{j=0}^{n}a_{ij}x_j$. We associate points $P_i \in \mathbb{P}^n$ with $L_i$ for all $1 \leq i \leq r$ given by $P_i = [a_{i0}: \cdots:a_{in}].$ The following lemma describes the deep connection between apolarity and Waring problems.  {We refer the reader to \cite[Lemma 1.15]{IK99} for a proof.}
    \begin{lemma}[Apolarity Lemma]\label{Lemma:Apolarity}
      Let $\mathbb{X}=\{P_1, \ldots, P_r\} \subset \mathbb{P}^n$ be a set of points. Then for some $\alpha_i \in \mathbb{C}$, we have $F = \sum_{i=1}^r\alpha_i L_i^d$ if and only if $I(\mathbb{X}) \subset F^{\perp},$ where $L_i$ is the linear form corresponding the point $P_i$ for $1 \leq i \leq r.$
    \end{lemma}
    
    Let {$\mathbb{X} \subset \mathbb{P}^n$} be a set of points. $\mathbb{X}$ is  said to be a  apolar to $F$ if $I(\mathbb{X}) \subset F^{\perp},$ and we call $I(\mathbb{X})$ an apolar ideal of $F.$  Further, $\mathbb{X}$ is a minimal apolar set of $F$ if $|\mathbb{X}| = \rk F.$ Note that the Apolarity Lemma is useful to give an upper bound for the Waring rank of $F.$

    In \cite{CCCGW18} the authors introduced the idea of e-computability, that is a technique to improve lower bounds for the Waring rank of a given form. An other result in this direction   is Theorem 1.3 in \cite{LT10}.


\begin{theorem}\label{Thm:Ecomputability}
    Let $F \in S_d,$ and $\mathbb{X}\in \mathbb{P}^n$ be an apolar set of $F.$ Let $I$ be an ideal in $T$ generated in degree $e >0,$  and $t \in I_e$. If $t$ is a non-zerodivisor in $T/I(\mathbb{X})\colon I$, then for $s >>0,$ we have
    \[ |\mathbb{X}| \geq \frac{1}{e}\sum_{i=0}^s\HF(T/I(\mathbb{X})\colon I +(t), i) \geq \frac{1}{e} \sum_{i=0}^s\HF(T/F^{\perp}\colon I +(t), i).\]
\end{theorem}
A homogeneous form $F\in S_d$ is said to be $e$-computable if there exists an ideal $I \subset T$, and $t \in I_e \neq 0$ such that 
\[\rk F = \frac{1}{e} \sum_{i=0}^s\HF(T/F^{\perp}\colon I+(t), i).\]

The study of the Waring locus of a form was introduced in \cite{ForbiddenLocus}. For $F \in S_d,$
the Waring locus  of $F$ is given by 
\[
\mathcal{W}_F = \{ P\in \mathbb{P}^n \mid P\in \mathbb{X}, \, I_{\mathbb{X}} \subset F^{\perp}, \, \mid \mathbb{X} \mid  =\rk F\,\}.
\]
The locus of forbidden points of a form is defined as its complement, that is $\mathcal{F}_F = \mathbb{P}^n \, \setminus \,\mathcal{W}_F.$

\section{Basic results}\label{Sec:Basic results}

For the first result of this section, we set $S=\mathbb{C}[x_1,\ldots,x_n]$ be the standard graded polynomial ring and let $T=\mathbb{C}[X_1,\ldots,X_n]$ be the ring of differential operators on $S$. The following lemma is, at the best of our knowledge, a new bound on the Hilbert function of a minimal apolar set to a given form.

\begin{lemma}\label{Lem:UpperboundHF}
    Let $\mathbb{X}$ be a minimal set of $s$ points apolar to $F\in S_d$ and let $V\subset (T/I({\mathbb{X}}))_a$ be a subspace. If $\bar{g}\in (T/I({\mathbb{X}}))_b$ where {b$\geq 0$,} is  a nonzerodivisor and $\bar{g}V\subset (F^\perp/I({\mathbb{X}}))$, then
    \[
        \HF(\mathbb{X},d-a-b)\leq s-\dim_\mathbb{C} V.   
    \]
\end{lemma}
\begin{proof}
    Let $\mathbb{X}$ be a set of $s$ points apolar to $F$.  For each $P_i = [a_{i,0}: \ldots :a_{i,n}] \in \mathbb{X}$ we can construct the linear form
$$L_i = a_{i,0}x_0+ a_{i,1}x_1 + \cdots + a_{i, n}x_n$$
such that, by the Apolarity Lemma \ref{Lemma:Apolarity},
\[F=\sum_{i=1}^s \epsilon_i L_i^d,\]
for some constants $\epsilon_i\in\mathbb{C}$. Note that, since $\mathbb{X}$ is minimal, we can assume the $\epsilon_i=1$ for $1\leq i\ \leq s$.
Let $\dim V=l$, and $v_1, v_2, \ldots, v_l$ be a basis of $V$. Since $\bar{g}V \subset F^{\perp}/I({\mathbb{X}})$, there exists $f_j \in I({\mathbb{X}})$ such that $f_j +gv_j \in F^{\perp}$.
Since $f_j \in I({\mathbb{X}}) \subseteq F^{\perp},$ we get  $gv_j \in F^{\perp}$ for $1\leq j \leq l$.

\noindent Hence for $1\leq j \leq l$,  $gv_j \circ F=0.$
Thus we get the system of equations 
\begin{eqnarray*}
    \sum_{i=1}^s gv_j(P_i) L_i^{d-a-b} =0
\end{eqnarray*}
for $1 \leq j \leq l$, Which in matrix form has the shape:

\begin{equation}\label{EQ:MatrixEQ}
    \begin{bmatrix}
        L_1^{d-a-b} & L_2^{d-a-b} & \ldots & L_s^{d-a-b}
    \end{bmatrix}
    \begin{bmatrix}
        gv_i(P_j) \\
    \end{bmatrix}_{1\leq i\leq l, 1\leq j\leq s}
    =\text{\Large0}.
\end{equation}
If  we let $\{u_1, \ldots, u_{c}\}$ be an ordered basis of $T_{d-a-b}$, where $c={\binom{n+d-a-b}{d-a-b}}$, the we can write $L_i^{d-a-b}$ as a linear combination of the chosen basis. Namely, there exists a $c\times s$ matrix $M$ such that
\begin{equation}\label{EQ:MatrixEQ2}
    \begin{bmatrix}
        L_1^{d-a-b} & L_2^{d-a-b} & \ldots & L_s^{d-a-b}
    \end{bmatrix}
    =
    \begin{bmatrix}
        u_1 & u_2 & \ldots & u_c 
    \end{bmatrix}
    M.
\end{equation}
Thus we have that
\begin{equation}\label{EQ:rkM}
\text{rank}(M) = \dim_\mathbb{C} \text{span} (L_1^{d-a-b} , L_2^{d-a-b},\ldots, L_s^{d-a-b})_{d-a-b}.
\end{equation}
Substituting \eqref{EQ:MatrixEQ2} in \eqref{EQ:MatrixEQ}, we obtain
\begin{equation*}
    \begin{bmatrix}
        u_1 & u_2 & \ldots & u_{c} 
    \end{bmatrix}
    M
    \begin{bmatrix}
    gv_i(P_j) \\
    \end{bmatrix}_{1\leq i\leq l, 1\leq j\leq s}
    =\text{\Large0}.
\end{equation*}
Since $u_1, \ldots, u_{c}$ is a basis of $T_{d-a-b}$, we get 
\begin{equation}\label{EQ:MN}
    M
    \begin{bmatrix}
         gv_i(P_j) \\
    \end{bmatrix}_{1\leq i\leq l, 1\leq j\leq s}
    =\text{\Large0}
\end{equation}
Moreover, we get that {\footnotesize{
\begin{equation*}
    \begin{bmatrix}
        gv_i(P_j) \\
    \end{bmatrix}_{1\leq i\leq l, 1\leq j\leq s}  = 
    GN
\end{equation*}
}}
where $G$ is the $s\times s$ diagonal matrix having on the diagonal $g(P_1),\ldots,g(P_s)$ and $$N =\begin{bmatrix}
         v_i(P_j)
    \end{bmatrix}_{1\leq i\leq l, 1\leq j\leq s}.$$
Note that, since $\bar{g}$ is a non-zerodivisor in $(T/I({\mathbb{X}}))_b$, $g(P_i) \neq 0$ for all $1 \leq i \leq s$, $G$
is invertible. 
Also, note that, since $\{v_1, v_2, \ldots, v_l\}$ are linearly independent in $T/I({\mathbb{X}})$, we have that
\begin{equation}\label{EQ:rkN}
   \text{rank}(N) =  \dim_k(V) = l.
\end{equation}
\noindent Thus Equation \eqref{EQ:MN} yields:
\begin{equation*}
    M GN =\text{\Large0}.
\end{equation*}
Therefore the \text{column space } of $GN$ is contained in the {null space } of $M$. Thus,
\[
\text{rank}(GN) \leq \nullity(M),
\]
where $\nullity(M)$ is the dimension of the right kernel of $M$.
Since $G$ is invertible, $GN$ and $N$ have the same column space. Thus 
\[
\text{rank}(N)=\text{rank}(GN) \leq \nullity(M)
\]
Thus, by the Rank-Nullity Theorem, we have
\[
    \text{rank}(M)=s-\nullity(M)
\]
and thus $\text{rank}(M)\leq s-\text{rank}(N)$.

By Equation \eqref{EQ:rkM}, we have
\begin{equation*}
        \text{rank}(M)  = \dim_\mathbb{C}(L_1^{d-a-b}, \ldots, L_{s}^{d-a-b}) 
         = \HF(\mathbb{X}, d-a-b). 
\end{equation*}
and, by Equation \eqref{EQ:rkN}, we have $\text{rank}(N) = \dim_{\mathbb{C}} V$.

Hence, the result is now proved.
\end{proof}

From now on, and throughout the paper, we consider a specialised version of the rings $S$ and $T$ in which we divide the variables into two set. Namely, we let $S=\mathbb{C}[x_1,\ldots,x_m,y_1,\ldots,y_n]$ be a standard graded polynomial ring and let $T=\mathbb{C}[X_1,\ldots,X_m,Y_1,\ldots,Y_n]$ the ring of differential operators on $S$.
\begin{lemma} \label{Lem:ProjectionOverSetOfVariablesSaturated}
    Let $\mathbb{X}$ be a finite set of points in $\mathbb{P}^{m+n-1}=\mathbb{P}(S_1)$ . If we let $\mathbb{Y}$ be the projection of $\mathbb{X}$ from $V(Y_1,\ldots,Y_n)$ on $V(X_1,\ldots,X_m)$, then
    \[
    I(\mathbb{Y})=(X_1,\ldots,X_m)+I(\mathbb{X})\cap\mathbb{C}[Y_1,\ldots,Y_n].
    \]
\end{lemma}

\begin{proof}
    Note that $I(\mathbb{X})$ is a radical ideal being an ideal of points, and so is $J=I(\mathbb{X}) \cap \mathbb{C}[Y_1,\ldots,Y_n]$ as the contraction of a radical ideal is radical {(see \cite[Exercise 1.18)]{AM69}}. Now we show that $(X_1,\ldots, X_m)+J$ is a radical ideal. Observe that
        \begin{equation}\label{EQ:reduced}
            \frac{\mathbb{C}[X_1,\ldots,X_m,Y_1,\ldots,Y_n]}{(X_1,\ldots, X_m)+J}\cong \frac{\mathbb{C}[Y_1,\ldots,Y_n]}{J}.
        \end{equation}
    Since $J$ is a radical ideal, ${\mathbb{C}[Y_1,\ldots,Y_n]}/{J}$ 
    is reduced. 
    Hence by Equation \eqref{EQ:reduced} we have $(X_1,\ldots, X_m)+J$ is a radical ideal. By the standard definition of projection, we have that $I(\mathbb{Y})$ is the saturation of $(X_1,\ldots, X_m)+J$. Since a radical ideal is saturated, the result is proved.
\end{proof}
Next lemma shows an interesting connection between apolar set of a polynomial and its projection.
\begin{lemma}\label{Lemma:ProjectionisApolar}
    Let $F=F_1F_2 \neq 0$ such that $F_1 \in S_1= \mathbb{C}[x_1, \ldots, x_m]$ and $F_2 \in S_2= \mathbb{C}[y_1, \ldots, y_n],$ and $\mathbb{X}$ be apolar to $F.$ If $\mathbb{Y}_1$ is a projection of $\mathbb{X}$ from $V(X_1, \ldots, X_m)$ to $V(Y_1, \ldots, Y_n)$, then $\mathbb{Y}_1$ is apolar to $F_1.$ Similarly, If $\mathbb{Y}_2$ is a projection of $\mathbb{X}$ from $V(Y_1, \ldots, Y_n)$ to $V(X_1, \ldots, X_m)$, then $\mathbb{Y}_2$ is apolar to $F_2.$ 
\end{lemma}
\begin{proof}
    By \cite[Proposition 3.1]{DFP20-ApolarAlgebra} we have 
    \begin{equation*}
        F^{\perp} = (F_1^{\perp})_{S_1} +(F_2^{\perp})_{S_2}
    \end{equation*}
  where  $(F_i^{\perp})_{S_i}$ denotes the annihilator of $F_i$ computed in $S_i$ for $i=1,2.$
  Let $\mathbb{Y}_1$ is a projection of $\mathbb{X}$ from $V(X_1, \ldots, X_m)$ to $V(Y_1, \ldots, Y_n)$. By Lemma \ref{Lem:ProjectionOverSetOfVariablesSaturated} we have 
  \begin{equation*}
      I(\mathbb{Y}_1) = I(\mathbb{X})\cap \mathbb{C}[X_1, \ldots, X_m] +(Y_1, \ldots, Y_n).
  \end{equation*}
  If $g \in I(\mathbb{X})\cap \mathbb{C}[X_1, \ldots, X_m],$ then $g\circ F = 0$. This implies $F_2 (g\circ F_1) = 0$. Since $F_2 \neq 0,$ we obtain  $g \in F_1^{\perp}.$
  Moreover, $(Y_1, \ldots, Y_n) \subseteq F_1^{\perp}$. Hence $I(\mathbb{Y}_1) \subset F_1^{\perp}$. Therefore $\mathbb{Y}_1$ is apolar to $F_1.$ Similarly we obtain $\mathbb{Y}_2$ is apolar to $F_2.$
\end{proof}
To study minimal apolar sets to a form $F$, we will use the Apolarity Lemma. Thus we need to describe the annihilators of the forms of interest for us. This is done in the following lemma.

\begin{lemma}\label{Lem:PerpF(x)G(y)}
    Let $F=(x_1^2+\cdots +x_m^2)(y_1^2+\cdots +y_n^2) ,$ where $m\geq 1, n\geq 2.$ Then 
    \[
F^{\perp} = 
\begin{cases}
  (X_1^3, \{Y_1^2-Y_i^2\}_{2\leq i\leq n},\{Y_iY_j\}_{1\leq i<j\leq n}) & \text{if } m=1 \\[6pt]
  (\{X_1^2-X_i^2\}_{2\leq i\leq m}, \{X_iX_j\}_{1\leq i<j\leq m},     \{Y_1^2-Y_i^2\}_{2\leq i\leq n},\{Y_iY_j\}_{1\leq i<j\leq n}) & \text{if } m\geq 2.
\end{cases}
\]
\end{lemma}
\begin{proof}
Let 
\[
I = 
\begin{cases}
  (X_1^3, \{Y_1^2-Y_i^2\}_{2\leq i\leq n},\{Y_iY_j\}_{1\leq i<j\leq n}) & \text{if } m=1 \\[6pt]
  (\{X_1^2-X_i^2\}_{2\leq i\leq m}, \{X_iX_j\}_{1\leq i<j\leq m},     \{Y_1^2-Y_i^2\}_{2\leq i\leq n},\{Y_iY_j\}_{1\leq i<j\leq n}) & \text{if } m\geq 2
\end{cases}
\]
and $T=\mathbb{C}[X_1,\cdots,X_m,Y_1,\cdots,Y_n].$ Note that $I\subseteq F^{\perp}.$ Moreover, a standard computations yields:

\[ \begin{array}{r|c|c|c|c|c}
     i&   0 & 1 & 2 &  3 & 4\\ 
     \hline 
     \HF(T/I, i)& 1 & n+m & nm+2 & n+m  & 1
\end{array} \]

Note that $T/F^{\perp}$ is an Artinian Gorenstein ring of socle degree 4. Thus, by symmetry, we have that 
$$\HF(T/F^{\perp},0)= \HF(T/F^{\perp},4)=1 \text{ and } \HF(T/F^{\perp},1)= \HF(T/F^{\perp},3)=n+m.$$
Now we compute $\HF(T/F^{\perp},2)$ by acting all possible degree 2 forms on $F.$ This can be done since apolarity gives a perfect pairing between $S_2$ and $T_2$ (see \cite{IK99}). For $1\leq i\leq m$ and $1\leq j\leq n,$ we have
\begin{align*}
   X_iY_j\circ F&=4x_iy_j \\
   X_i^2 \circ F&=y_1^2+\cdots+y_n^2\\
   Y_j^2 \circ F&=x_1^2+\cdots+x_m^2.
\end{align*}
Thus, $\HF(T/F^{\perp},2)=mn+2,$ and $T/I$ and $ T/F^\perp$ have the same Hilbert functions. Hence, since $I \subseteq F^\perp,$ we showed that $I=F^{\perp}.$
\end{proof}

    We will often use group actions to study apolar subsets. In particular, given the ring
    $$\mathbb{C}[x_1,\ldots,x_m,y_1,\ldots,y_n]$$
    we will consider the following action. We denote by $O(n)$ the subgroup of the complex orthogonal group $O(\mathbb{C},m+n)$ acting as the identity on the $m$ variables $x_i$ and in the standard way on the $n$ variables $y_j$. To a given linear form
    $$L=a_1x_1+\cdots a_mx_m+b_1y_1+\cdots +b_ny_n$$
    corresponds  a projective point
    $$L_P=[a_1:\ldots:a_m:b_1:\ldots:{b_n}],$$ and 
    the action of $g\in O(n)$ on $L$ corresponds to the action of $g^T$ on $L_P$. Moreover, if $\mathbb{X}$ is apolar to a form $F$ and $g\in O(n)$, then $g^T(\mathbb{X})$ is apolar to $g(F)$. 

In particular, we will use the following technical lemma to avoid zeroes in apolar subsets.
\begin{lemma}\label{Lem:NonZeroCoordinatesAfterRotation}
    Let $F\in\mathbb{C}[x_1,\ldots,x_m,y_1,\ldots,y_n]$ be a form fixed by the action of $O(n)$, and let $\mathbb{X}$ be a set of $r$ points apolar to $F$. Then, there exists a set $\mathbb{Y}$ of $r$ points apolar to $F$ such that, if
    $$[a_1:\ldots:a_m:b_1:\ldots:b_n]\in\mathbb{Y},$$
    then $b_j\neq 0$ for all $j$.
\end{lemma}
\begin{proof}
For $P\in\mathbb{X}$ and $g\in O(n)$ we set
$$g^T(P)= [a_1:\ldots:a_m:b_1:\ldots:b_n]$$
and we define
$$U(P,j)=\{g\in O(n) \mid  b_j=0 \}.$$
Note that $U(P,j)$ is defined by a linear equation in $O(n)$. In particular, $U(P,j)$ is a proper subvariety of $O(n)$.
If we let
$$U=\bigcup_{P\in\mathbb{X},1\leq j\leq n} U(P,j),$$
then $U$ is a proper subvariety of $O(n)$. Let $g\in O(n)\setminus U$ and let
$$\mathbb{Y}=g^T(\mathbb{X}).$$
By construction $\mathbb{Y}$ is a set of $r$ points with the required property and the proof is now completed.
\end{proof}

We will often deal with finite set of points in general linear position, that is each subset has the maximal possible linear span. Thus we recall the following basic fact:
\begin{remark}\label{Lem:IdealOfNpointsGeneralPosition}  If $\mathbb{Y} = \{P_1,\ldots,P_n\}$ be a set of $n$ points in general position in $\mathbb{P}^{n-1},$ then there exist linearly independent linear forms $L_1,\ldots,L_n$  such that 
        \[
        I(\mathbb{Y}) = (L_iL_j \mid 1 \leq i <j \leq n).
        \]
    
\end{remark}

{
Next we prove a useful consequence of e-computability.
\begin{lemma}\label{Lemma:eComputability&ForbiddenLocus}
   If $F$ is e-computable by an ideal $I$, and $t \in I_e\neq 0$, then $V(I) \subseteq \mathcal{F}_F.$ 
\end{lemma}
\begin{proof}
    By the definition of e-computability, for any minimal apolar set $\mathbb{X}$ to $F$, $t$ is a non-zerodivisor on $T/I(\mathbb{X}):I.$ Hence we have the following exact sequence:
    \[
    0 \xrightarrow{\quad} \left(\frac{T}{I(\mathbb{X}):I} \right)_{i-e} \xrightarrow{\quad \cdot t \quad}  \left(\frac{T}{I(\mathbb{X}):I}\right)_i \xrightarrow{\quad\quad}  \left(\frac{T}{I(\mathbb{X}):I +(t)}\right)_i \xrightarrow{\quad} 0.
    \]
    Therefore, for $s \gg 0,$ we have 
  \begin{equation}\label{EQ:YnumberPoints}
      e \, \cdot \HF\left(\frac{T}{I(\mathbb{X}):I}, s\right) = \sum_{i=0}^s \HF\left(\frac{T}{I(\mathbb{X}):I+(t)}, i \right)
  \end{equation}
   Note that $I(\mathbb{X}):I$ is a saturated ideal. Let $\mathbb{Y} = V(I(\mathbb{X}):I)$ that is the set of points in $\mathbb{X}$ that do not lie on $V(I)$. Thus $\mathbb{Y}\subseteq \mathbb{X}$. 
   Moreover, for $s \gg 0,$ 
   \begin{equation}\label{EQ:Y}
       \mid \mathbb{Y} \mid = \HF\left(\frac{T}{I(\mathbb{X}):I}, s\right).
   \end{equation}
   Since $F$ is e-computable by $I$ and $t \in I_e$  for $s \gg 0,$ we have
   \begin{equation}\label{Eq:eCompX}
       \mid \mathbb{X} \mid = \rk F = \frac{1}{e}\sum_{i=0}^s \HF\left(\frac{T}{I(\mathbb{X}):I+(t)}, i \right) =\frac{1}{e} \sum_{i=0}^s \HF\left(\frac{T}{F^{\perp}:I+(t)}, i \right) 
   \end{equation}
   for any minimal apolar set $\mathbb{X}$ to $F.$
   By Equations \eqref{EQ:Y} and \eqref{Eq:eCompX} we obtain
   $\mathbb{Y} = \mathbb{X}.$
   Therefore $\mathbb{X} \cap V(I) = \emptyset$ for any minimal apolar set $\mathbb{X}$ to $F.$ Hence we have proved the result.
\end{proof}
}

 {In the next proposition we compute the Waring rank of the forms of the type  $F = x_1^{a_1}x_2^{a_2} \cdots x_m^{a_m}(y_1^b+\cdots+y_n^b),$ where $m,n \geq 1,$ $ b \geq 2$, and $a_1 \leq a_2 \leq \cdots \leq a_m$ with $a_1+1 \geq b$. Note that this is includes the case $F= (x_1^2+x_2^2)(y_1^2+\cdots+y_n^2)$.}

{\begin{proposition}\label{Prop:RankOfForms}
     Let $F = x_1^{a_1}x_2^{a_2} \cdots x_m^{a_m}(y_1^b+\cdots+y_n^b),$ where $m,n \geq 1,$ $ b \geq 2$, and $a_1 \leq a_2 \leq \cdots \leq a_m$ with $a_1+1 \geq b$. Then 
     \[ 
     \rk F = (\prod_{i =1}^m (a_i + 1))n.
     \]
\end{proposition}}
\begin{proof} 
        We use $e$-computability to obtain a lower bound.
        Consider $I = (Y_1, Y_2, \ldots, Y_n)$ and $t = \sum_{i=1}^n \alpha_iY_i$. Let $\tilde{I} = F^{\perp} : I +(t).$
        Then 
        \begin{equation}\label{EQ:tildeIrank}
            \tilde{I} = (\{X_i^{a_i+1}\}_{1\leq i \leq n}, \{Y_iY_j\}_{1\leq i < j \leq n}, \{Y_i^b\}_{1 \leq i \leq n}, \sum_{i=1}^n \alpha_iY_i).
        \end{equation}
        We want to compute the Hilbert function of $T/\tilde{I}$. For $l \geq 0$, and $l+1 <b,$
        
    \begin{multline}\label{Eq:HFComputation}
        \HF(T/\tilde{I}, l+1) = (\text{No. of linearly independent degree } l\text{ monomials in } T/\tilde{I} \text{ involving only } x_i'\text{s})(n-1) \\ + (\text{No. of linearly independent degree } l+1 \text{ monomials in } T/\tilde{I} \text{ involving only } x_i'\text{s}).
    \end{multline}
            Let ${\bf{a}} = (a_1, a_2, \ldots, a_n)$. Define $$C(m ,s, {\bf{a}}) = \mid \{ (l_1, l_2, \ldots, l_m) \mid  \sum_{i=1}^m l_i = s \text{ and } l_i \leq a_i\} \mid.$$ 
            Here $C(m ,s, {\bf{a}})$ counts the number of linearly independent monomials of degree $s$ in $T/\tilde{I}$ involving only $x_i$'s. By Equation \eqref{Eq:HFComputation}, for all $l\geq 0$,
            \begin{equation}\label{Eq:HFComb}
                \HF(T/\tilde{I}, l+1 ) = C(m, l, {\bf a})(n-1) + C(m, l+1, {\bf a}).
            \end{equation}

    Further, $\HF(T/\tilde{I}, (\sum_{i=1}^m a_i)+1) = 0.$ Let $\alpha =\sum_{i=1}^m a_i$.
    We have

\begin{equation}\label{EQ:RankLowerBound}
    \begin{split}
        \sum_{l =0}^{\alpha+1} \HF(T/\tilde{I}, l)&  = 1 +  \sum_{l =1}^{\alpha+1} \HF(T/\tilde{I}, l) = 1 + \sum_{l =0}^{\alpha} \HF(T/\tilde{I}, l+1)\\
        & = 1 + \sum_{l =0}^{\alpha}  \left(C(m, l, {\bf a})(n-1) + C(m, l+1, {\bf a})\right) \quad (\text{by Equation } \eqref{Eq:HFComb}) \\
        & =   n \sum_{l =0}^{\alpha}   C(m, l, {\bf a}) = n \prod_{i=1}^m (a_i+1)
        \end{split}
\end{equation}
   Therefore, by Equation \eqref{EQ:RankLowerBound}, 
   \begin{equation}\label{EQ:rankEq1}
       \rk F \geq \sum_{l \geq 0} \HF(T/ \tilde{I}, l) = n \prod_{i=1}^m (a_i+1),
   \end{equation}
  where the first inequality holds by \cite[Theorem 3.3]{CCCGW18}.
   Further, the ideal 
   \[ I({\mathbb{X}}) = ( \,\{Y_iY_j\}_{1 \leq i <j \leq n}, \, Q_1, Q_2, \ldots, Q_m)\] where $Q_i = X_i^{a_i+1} + (Y_1^{a_i+1-b} + \cdots +Y_n^{a_i+1-b})((-n+1)Y_1^b+ Y_2^b+\cdots+Y_n^b)$ is contained in $F^{\perp}$. We have 
   \[
   I({\mathbb{X}}) = \bigcap_{i=1}^n \, (Y_1, Y_2, \cdots, \hat{Y_i} \cdots Y_n, \, q_1, \, q_2, \ldots, \, q_m), 
   \]
   where $q_j = X_j^{a_j+1}+Y_i^{a_i+1}$ for all $1 \leq j \leq m$, for all $2 \leq i \leq n$, and for $i=1,$ for all $1 \leq j \leq m$, $q_j = X_j^{a_j+1}+(-n+1)Y_i^{a_i+1}$. Note that over the field of complex numbers, $q_j$ split into linear forms, giving the primary decomposition of $I({\mathbb{X}})$. Hence $I({\mathbb{X}})$ is an ideal of $n \prod_{i=1}^m (a_i+1)$ distinct points.
   Therefore by the Apolarity Lemma \ref{Lemma:Apolarity}, 
   \begin{equation}\label{EQ:rankEq2}
       \rk F \leq n \prod_{i=1}^m (a_i+1).
   \end{equation}
  Therefore, by inequalities \eqref{EQ:rankEq1} and \eqref{EQ:rankEq2}, we have obtained the result.
\end{proof}

\section{The form $x^2(y_1^2+\cdots+y_n^2)$}
In this section, we focus on the case  $F=x^2(y_1^2+\cdots+y_n^2)$ with $n\geq 2$. Thus we consider the rings $S=\mathbb{C}[x,y_1,\ldots,y_n]$ and $T=\mathbb{C}[X,Y_1,\ldots,Y_n]$. We recall that, by {\cite[Proposition 4.4]{CCCGW18}}, we know that $\rk F = 3n$.

We will now proceed to investigate all minimal apolar sets of $F$. Our results are visually summarized in Figure \ref{Picture:PointsOnLines}.

\begin{proposition}\label{Prop:structurex^2}
    Let $F=x^2(y_1^2+\cdots+y_n^2)$ with $n\geq 2$ and let $\mathbb{X}$ be a set of $\rk F=3n$ distinct points apolar to $F$. Then the following holds:
        \begin{enumerate}
            \item[i)] If $\mathbb{Y}$ is the projection of $\mathbb{X}$ from $V(Y_1,\ldots,Y_n)$ on $V(X)$, then $\mathbb{Y}$ is a set of $n$ distinct points in linear general position, that is
                \[
                \begin{array}{r|c|c|c}
                i & 0 & 1 & 2 \\
                \hline
                \HF(\mathbb{Y},i) & 1 & n & n \\
                     
                \end{array},
                \]
                    and $\HF(\mathbb{Y},i)=n$ for $i\geq 3$.
            \item[ii)] Let $\mathbb{Y}=\{Q_1,\ldots,Q_n\}$ and $P=[1:0: \ldots :0]$. If $\Lambda_i=\langle  Q_i,P \rangle,$  where $\langle Q_i,P\rangle$ denotes the linear span of $Q_i$ and $P$, then $\mathbb{X}\cap\Lambda_i$ is a set of three distinct points for $1\leq i\leq n$. \end{enumerate}

\end{proposition}
\begin{proof} We start by proving (i). By Lemma \ref{Lem:PerpF(x)G(y)}  we have
            \begin{equation}\label{EQ:PerpF}
                F^{\perp} = (X^3, \{Y_iY_j\}_{(1\leq i < j \leq n)}, \{Y_1^2-Y_i^2\}_{(2 \leq i \leq n)}).
            \end{equation}
    Let $I({\mathbb{X}})$ be an ideal of a minimal apolar set $\mathbb{X}$ of $F$.  Since $O(n)$ fixes $F$, by Lemma \ref{Lem:NonZeroCoordinatesAfterRotation} we may assume that  for $P_i \in \mathbb{X}$, $P_i = [b_i:a_{i,1}: \ldots :a_{i,n}]$ is such that $a_{i,1} \neq 0$ for $1\leq i \leq {3n}$.  Hence $Y_1$ is a non-zerodivisor in $T/I(\mathbb{X})$. Let $V := \langle Y_2, Y_3, \ldots, Y_n\rangle_{\mathbb{C}} \subset (T/I({\mathbb{X}}))_1$. Since $I(\mathbb{X}) \subset F^{\perp},$ we have that $I(\mathbb{X})_1 = 0.$ Thus $\dim V = n-1$, and $Y_1V \subset (F^{\perp}/I({\mathbb{X}}))$. Hence, by Lemma \ref{Lem:UpperboundHF}, we obtain that
        \begin{equation}\label{EQ:HFUB2n+1}
            \HF(\mathbb{X}, 2) \leq 3n - (n-1) = 2n +1,
        \end{equation}
    and that
        \begin{equation}\label{EQ:IXinDim2}
            \dim \, I(\mathbb{X})_2 \geq \binom{n+2}{2}-(2n+1)=\binom{n}{2}.
        \end{equation}
    Note that, by Equation \eqref{EQ:PerpF}, $F^{\perp}$
        does not involve the variable $X$ in any degree 2 terms.
   Thus we get
        \begin{equation}\label{EQ:IXEqualsIYDegree2}
            I({\mathbb{X}})_2 = (I(\mathbb{X}) \cap \mathbb{C}[Y_1, \ldots, Y_n])_2.
        \end{equation}
     Therefore by Equation \eqref{EQ:IXinDim2} we obtain
         \begin{equation}\label{EQ:DimIXcapYinDegree2}
             \dim (I(\mathbb{X}) \cap \mathbb{C}[Y_1, \ldots, Y_n])_2 \geq \binom{n}{2}.
         \end{equation}

    \noindent Let $\mathbb{Y}$ be the projection of $\mathbb{X}$ from the point $P=[1:0:\ldots:0]$ on the hyperplane $V(X)$. By Lemma \ref{Lem:ProjectionOverSetOfVariablesSaturated} we have
    \[I({\mathbb{Y}}) =(X)+ I({\mathbb{X}}) \cap \, \mathbb{C}[Y_1, \ldots, Y_n].\]
    Since $I({\mathbb{X}})_1 = 0,$ we get $I({\mathbb{Y}})_1=(X)_1$ and thus
     \[\HF(\mathbb{Y},1)=n.\]
    Since $I({\mathbb{Y}})_2= (X)_2+(I({\mathbb{X}}) \cap \mathbb{C}[Y_1, \ldots, Y_n] )_2$,  Equation \ref{EQ:DimIXcapYinDegree2} yields
        \begin{equation}
            \HF(\mathbb{Y},2)\leq\binom{n+2}{2}-\left((n+1)+\binom{n}{2}\right)=n.
        \end{equation}
    Since $\HF(\mathbb{Y},1)=n$ and $I(\mathbb{Y})$ is an ideal of reduced points, {by \cite{GMR83} we have
    \[
    \Delta \HF(\mathbb{Y}) = (1, \HF(\mathbb{Y}, 1) -  \HF(\mathbb{Y}, 0), \ldots,  \HF(\mathbb{Y}, i) -  \HF(\mathbb{Y}, i-1), \ldots )
    \]
    is an O-sequence. Hence we obtain $\HF(\mathbb{Y},2)= n.$ Moreover, by Macaulay's criterion (see \cite[Theorem 4.2.10]{BH98}) for O-sequences, $\Delta \HF(\mathbb{Y}, 1) = 0$ implies $\Delta \HF(\mathbb{Y}, l) = 0$ for all $l \geq 1.$} Hence we get the Hilbert function of $\mathbb{Y}$ as 
            \[
            \begin{array}{r|c|c|c|c}
                i & 0 & 1 & 2 & \cdots \\
                \hline
                \HF(\mathbb{Y},i) & 1 & n & n& n \\
                     
            \end{array}
                \]
    Therefore $\mathbb{Y}$ is a set of $n$ points in a general position. This concludes the proof of (i).

\vskip 2mm

\noindent We now prove (ii). 
Let $\mathbb{Y} = \{Q_1,\ldots,Q_n\}$. We have 
        \begin{equation*}
            \mathbb{X} \subset \bigcup_{i=1}^n \langle \, Q_i,P\rangle.
        \end{equation*}
    Let $\mathcal{H}_i= \langle P,Q_1,\ldots, \hat{Q_i} ,\ldots, Q_n \rangle$ be the hyperplane generated by the linear span of $P$ and the points $Q_1,\ldots, \hat{Q_i} ,\ldots, Q_n.$ Note that, since $(F^\perp)_1=0$, $Q_i\notin \mathcal{H}_i$.
    If we let $L_i=\sum_{j=1}^n \alpha_{ij}Y_j.$ then 
        \begin{equation*}
                \begin{split}
                    F^{\perp} : I(\mathcal{H}_i)=F^{\perp} : L_i & = ((\alpha_{i1} Y_1 +  \cdots + \alpha_{in} Y_n) \circ F)^{\perp} \\
                    & = (2x^2(\alpha_{i1} y_1 +  \cdots + \alpha_{in} y_n))^{\perp},
                \end{split}
        \end{equation*}
    where the second equality holds by \cite[Lemma 3.2]{CCCGW18}. Note that 
         \begin{equation*}
           I(\mathbb{X} \cap \Lambda_i) \subseteq   I(\mathbb{X}):I(\mathcal{H}_i) \subseteq F^{\perp} : I(\mathcal{H}_i)=G^{\perp},
        \end{equation*}
    where $G=2x^2(\alpha_{i1} y_1 +  \cdots + \alpha_{in} y_{n}).$ {After a change of variables, $G$  is a binary monomial and hence $\rk G= 3$ by Sylvester's algorithm (see \cite{sylvester1851lx, Sylvester1851} or \cite[Remark 4.16]{CGO14}).}   Therefore each of the $n$ lines $\Lambda_i$ contains at least $3$ points of $\mathbb{X}.$ Note that $P$ is the only point in the intersection of any two lines and that $P \in \mathcal{F}_F$ by Lemma \ref{Lemma:eComputability&ForbiddenLocus}. Thus we have at least $3n$ points in the union of $n$ lines. Hence, since $\mid \mathbb{X} \mid = 3n$, each line contains exactly $3$ points. This concludes the proof of (ii).

{The proof is now completed}.
    
\end{proof}

We can now characterize the points that appear in a minimal apolar set to $F$.
\begin{theorem}\label{Thm:forbiddenLocusofX^2sumofSquares}
    If $F=x^2(y_1^2+\cdots+y_n^2)$, then the forbidden locus of $F$ is given by 
        \[
        \mathcal{F}_F=V(Y_1^2+\cdots+Y_n^2).
        \]
\end{theorem}
\begin{proof} Let $Q \in \mathcal{W}_F$ be a point of the Waring locus of $F$, that is there exists $\mathbb{X}$, a minimal apolar set to $F$, such that $Q\in\mathbb{X}$. Let  $\mathbb{Y}$ be the projection of $\mathbb{X}$ from $P=[1:0: \ldots :0]$ into $V(X)$. By Proposition \ref{Prop:structurex^2} we know that $\mathbb{Y}$ is a set $n$-points in general position. Thus $\mathbb{Y}$ is a minimal apolar set of $G=y_1^2+ \cdots +y_n^2.$  By Proposition \cite[Proposition 3.1]{ForbiddenLocus} we have $\mathcal{F}_G = V(Y_1^2+ \cdots +Y_n^2)$, thus $Q \notin V(Y_1^2+\cdots +Y_n^2).$ Hence we proved that
\[
\mathcal{F}_F \supseteq V(Y_1^2+ \cdots +Y_n^2).
\]
    
    We will now show that equality holds, that is, we prove that if $Q \notin V(Y_1^2+ \cdots Y_n^2)$, then $Q \in \mathcal{W}_F.$ We split the proof in two cases depending on the first coordinate corresponding to $X$ in $Q$ to be zero or not.\\
{
If the first coordinate of $Q$ is zero, then consider 
\[ I(\mathbb{X})= (\{Y_iY_j\}_{1 \leq i<j\leq n}, X^3+X(Y_1^2+ \cdots+Y_{n-1}^2+(-n+1)Y_n^2)).\]
Since the orthogonal action of $O(n)$ on the variables $Y_1, \ldots, Y_n$ with $\sum_{i=1}^n Y_i^2 \neq 0$, that is the action on $\mathbb{P}^{n-1} \setminus V(\sum_{i=1}^n{Y_i^2})$ has only one orbit, there exists $\sigma \in O(n)$ such that $\sigma(\mathbb{X}) = \mathbb{X}_1$ with $Q \in \mathbb{X}_1.$
\newline  If the first coordinate of $Q$ is non-zero, then consider
\[I(\mathbb{X}') = (\{Y_iY_j\}_{1 \leq i<j\leq n}, X^3+(- \sum_{i=1}^nY_i)(Y_1^2+ \cdots+Y_{n-1}^2+(-n+1)Y_n^2)).\]
Again, by the similar arguments above we get $Q \in \mathbb{X}_2$ apolar to $F$. Since $F$ is fixed by the action of orthogonal groups $O(n)$, the Waring locus and the forbidden locus are stable under the action of orthogonal group. Hence  result follows.
}
\end{proof}

We now focus on algebraic properties of minimal apolar sets to $F$: we begin by describing their Hilbert functions.

\begin{proposition}\label{Prop:HFofX2SumofSquares}
    If $F = x^2(y_1^2+ \cdots +y_n^2)$, $n \geq 2$, and $\mathbb{X}$ is a set of $3n$ points apolar to $F$, then the Hilbert function of $\mathbb{X}$ is
    \[
    \begin{array}{r|c|c|c|c|c|c}
        i & 0 & 1& 2&3 & 4& \cdots  \\
        \hline
        \HF(\mathbb{X}, i) & 1 & n+1 & 2n+1  & 3n & 3n & 3n.
    \end{array}
    \]
\end{proposition}

\begin{proof}
    Since $I(\mathbb{X}) \subset F^{\perp}$ and  $(F^\perp)_1 = 0$ we obtain $\HF(\mathbb{X},1)=n+1$. Moreover, we claim that $\HF(\mathbb{X},2)= 2n+1$.
    \begin{proof}[Proof of the claim]
    {Note that by Equation \eqref{EQ:HFUB2n+1} we have that $\HF(\mathbb{X},2) \leq 2n+1.$} By contradiction we assume that $\HF(\mathbb{X},2)\leq 2n$ and we use the notations of the proof of Proposition \ref{Prop:structurex^2}. From the assumption we get that $\dim I({\mathbb{X}})_2 \geq \binom{n+2}{2} - 2n.$ Since $F^{\perp}_2$ does not involve any forms in the variable $X,$ so is $I({\mathbb{X}})_2$. Therefore, $\dim I({\mathbb{X}})_2 \cap \mathbb{C}[Y_1,Y_2, \ldots, Y_n] \geq \binom{n+2}{2} - 2n.$
    
    We obtain $\HF(\mathbb{Y}, 2) \leq n-1$, a contradiction.
    \end{proof}
    We now prove by contradiction that $\HF(\mathbb{X},3)=3n$. Since $\mathbb{X}$ is a set of $3n$ points, we have that $\HF(\mathbb{X},3)\leq 3n$. Thus, we assume by contradiction that $\HF(\mathbb{X},3)\leq 3n-1$. Let $W=I(\mathbb{X})_3$. Then
    \begin{equation}\label{Equation:dimW}
          \dim W=\dim I(\mathbb{X})_3\geq \binom{n+3}{3}-3n+1.
    \end{equation}
    By Proposition \ref{Prop:structurex^2}(ii) we know that $\mathbb{X}$ is supported on $\mathcal{L}$ where $\mathcal{L}$ is the union of $n$ lines passing through the point $P=[1:0:\ldots:0]$ and each line in $\mathcal{L}$ contains exactly $3$ points of $\mathbb{X}$.
    Let $W' \subseteq W$ such that every element of $W'$ passes through the point $P$. Hence 
    \begin{equation}\label{Equation:dimW'}
        \dim W'\geq \dim W-1.
    \end{equation}
    Note that $W' \subset T_3$ and that each of its elements vanishes on $4$ points of each line of $\mathcal{L}$. Thus $W'$ vanishes on $\mathcal{L}.$
    Let $W'' \subset W'$ be the set of all elements passing through $\binom{n+2}{3}-n$ general points of the hyperplane $V(X)$. Note that
    Moreover,
        \begin{equation}\label{EQ:DimV''}
            \begin{split}
                \dim W'' & \geq \dim W'-\binom{n+2}{3}+n  \geq \binom{n}{2}+1 
            \end{split}
        \end{equation}
    where the last inequality follows from Equations \eqref{Equation:dimW} and \eqref{Equation:dimW'}. Since $\mathcal{L}\cap V(X)$
    consists of $n$ points, each element of $W''$ has $X$ as a factor.
    Thus, we obtain
    \[W''=\langle XQ_1,\ldots XQ_j\rangle\]
    where $j\geq\binom{n}{2}+1$ and the forms $Q_i$ are linearly independent degree two forms. Since $V(X)\not\supset\mathcal{L},$ we have 
    \[
    V(Q_1,Q_2,\ldots,  Q_j)\supset\mathcal{L}
    \]
    for some $j\geq\binom{n}{2}+1$. Thus, $Q_1,\ldots,Q_j\in I(\mathcal{L})$. Note that $\mathcal{L}$ is the non-degenerate union of $n$ lines through $P$, thus
    \[
    I(\mathcal{L})=(L_iL_j:1\leq i<j \leq n),\]
    where the linear forms $L_i$ do not involve $X$.
    In particular, the ideal of $\mathcal{L}$ is generated by $\binom{n}{2}$ degree two forms. Thus a contradiction, since the forms $Q_i$ are linearly independent and hence $\HF(\mathbb{X}, 3) = 3n.$
    Since ${\mathbb{X}}$ is a set of $3n$ reduced points, the proof is completed.
\end{proof}

We now complete the algebraic description of minimal apolar sets to $F$ by describing their defining ideals.
{
\begin{proposition}\label{Prop:ShapeofIdealMonomialproductSumofSquares}
     Let $F=x^2(y_1^2+\cdots+y_n^2)$ with $n \geq 2.$ If $\mathbb{X}$  is a minimal apolar set of $F$, then there exists $\sigma \in O(n)$ such that the ideal of $I(\sigma(\mathbb{X}))$ is given by
    \[I(\sigma(\mathbb{X}))= (\{Y_iY_j\}_{1\leq i < j \leq n}, \, X^3+\sum_{i=1}^n L_i(Y_1^2-Y_i^2))\]
    for some linear form $L_i $ in $\mathbb{C}[X,Y_1, \ldots, Y_n]$ for $1 \leq i \leq n.$
\end{proposition}}
\begin{proof}
    By Proposition \ref{Prop:HFofX2SumofSquares}, we have $\HF(\mathbb{X}, 3) = 3n$.
     Keeping the same notation as in Proposition \ref{Prop:HFofX2SumofSquares}, 
     we have $\HF(\mathbb{X},2) = \binom{n}{2}$ and $\mathbb{Y}$ is an ideal of $n$ points in general position apolar to $G=y_1^2+ \cdots+y_n^2.$ 
     {By \cite[Proposition 3.1]{ForbiddenLocus} for any point $P\in \mathbb{X}$, we have $P \notin \mathcal{F}_G$. Thus, we obtain $P \notin V(\sum_{i=1}^nY_i^2)$ for any point $P\in \mathbb{X}$. Since the action of orthogonal group $O(n)$ on the variables $Y_1, \ldots, Y_n$ with $\sum_{i=1}^nY_i^2 \neq 0$ has only one orbit, there exists $\sigma \in O(n)$ such that $\mathbb{Y} = \{E_1, E_2, \ldots, E_n\}$ where $E_i$ is the standard $i^{th}$ coordinate point for $1 \leq i \leq n$}. Hence
     \[(Y_iY_j \mid 1 \leq i < j \leq n) \subseteq I(\sigma(\mathbb{X})) \cap \mathbb{C}[Y_1, \ldots, Y_n].\] Therefore 
     \[I(\sigma(\mathbb{X}))_2 =(Y_iY_j \mid 1 \leq i < j \leq n).\]
     Further, by Proposition  \ref{Prop:HFofX2SumofSquares} we have $\sigma(\mathbb{X}) \subset \mathcal{L}.$ Hence $I (\sigma(\mathbb{X})) \supset I(\mathcal{L}).$
     Note that $\HF(\mathcal{L}, 3)= \binom{n+3}{3}-(3n+1).$ 
     Hence, there exists exactly one degree 3 generator in $I(\sigma(\mathbb{X})).$ Since $I(\sigma(\mathbb{X}) )\subset F^{\perp}$, a degree 3 generator has to be of the form $X^3+\sum_{i=2}^{n}L_i(Y_1^2-Y_i^2)$ for some $L_i \subset T_1$, $1 \leq i \leq n$. We write 
     \[I(\sigma(\mathbb{X})) = \bigcap_{i =1}^n \, (Y_1, Y_2, \ldots,\hat{Y_i}, \ldots, Y_n, \, Q= X^3+\sum_{i=2}^{n}L_i'(Y_1^2-Y_i^2)\,)+J,\]
    where $J$ is generated in degree at least four in $I(\sigma(\mathbb{X})).$ 
    By Proposition \ref{Prop:structurex^2}, we have 3 points on each line $V(Y_1, \, \ldots,\hat{Y_i}, \ldots, Y_n)$ for all $1 \leq i \leq n.$ Consider 
    $\bar Q \in \frac{\mathbb{C}[X, Y_1, \ldots, Y_n]}{(Y_1, \ldots,\hat{Y_i}, \ldots, Y_n)}.$ 
    Note that $\bar Q = X^3 +(\alpha X+kY_i)Y_i^2$ with $(\alpha, k) \neq 0$, and $\bar Q$ does not have a multiple factor. Indeed, if $\bar Q$ has a multiple factor, then there exists a degree 2 form in $(\bar Q)+ \bar J$. Since  $(\bar Q)+ \bar J$ is an ideal of 3 points, this is a contradiction.
    Hence $\bar Q$ does not have a multiple factor. Therefore, $(\bar Q)$ is an ideal of $3$ points. Thus $V(\bar Q) =  V( (\bar Q )+ \bar J)$. Therefore $(\bar Q) \supseteq \bar J.$ Hence $I(\sigma(\mathbb{X}))$ has generators in degree at most 3.
\end{proof}

Theorem \ref{thm:A} is now proved by Propositions \ref{Prop:structurex^2}, \ref{Thm:forbiddenLocusofX^2sumofSquares}, \ref{Prop:HFofX2SumofSquares}, \ref{Prop:ShapeofIdealMonomialproductSumofSquares}.

\section{The form $(x_1^2+x_2^2)(y_1^2+\cdots+y_n^2)$}
We note that up to a change of coordinates the forms 
$(x_1^2+x_2^2)(y_1^2+\cdots+y_n^2)$ and $x_1x_2(y_1^2+\cdots+y_n^2)$ can be identified. Hence, throughout this section, we assume the form to be $F=
x_1x_2(y_1^2+\cdots+y_n^2)$ with $n\geq 2$. By  Proposition \ref{Prop:RankOfForms}, we obtain that $\rk F = 4n.$ We the analyse all minimal apolar sets of $F$, we obtain the forbidden locus of $F$, and we also describe their ideals and Hilbert functions.

We begin by studying the projection of any minimal apolar set from $V(Y_1,\ldots,Y_n)$. A summary of the below result can be visualized in Figure \ref{Picture:PointsOnPlane}.

 \begin{proposition}\label{prop:projfromline}
     Let $F=x_1x_2(y_1^2+\cdots+y_n^2)$ with $n \geq 2$ and let $\mathbb{X}$ be a minimal apolar set of points to $F$. Then $\rk F=4n$ and the following facts hold:
     \begin{enumerate}
         \item[i)]   If $\mathbb{Y}$ is the projection of $\mathbb{X}$ from $V(Y_1,\ldots,Y_n)$ on $V(X_1,X_2)$, then $\mathbb{Y}$ is a set of $n$ distinct points in linear general position, that is
        \[
        \begin{array}{r|c|c|c}
        i & 0 & 1 & 2 \\
        \hline
        \HF(\mathbb{Y},i) & 1 & n & n \\
             
        \end{array},
        \]
        and $\HF(\mathbb{Y},i)=n$ for $i\geq 3$.
        
        \item[ii)] Let $\mathbb{Y}=\{Q_1,\ldots,Q_n\}$ and let $\ell=V(Y_1,\ldots,Y_n)$. If we consider the plane $\Lambda_i=\langle \ell, Q_i \rangle$, where $\langle \ell, Q_i \rangle$ denotes the linear span of $\ell$ and $Q_i,$ then $\mathbb{X}\cap\Lambda_i$ is a set of four distinct points for $1\leq i\leq n$. 
       
     \end{enumerate}
 \end{proposition}
\begin{proof} Proposition \ref{Prop:RankOfForms} immediately yields that $\rk F=4n$ we use.
We now prove (i).
     By Lemma \ref{Lem:PerpF(x)G(y)} we have
    \begin{equation}\label{EQ:PerpF2}
         F^{\perp} = (X_1^2,X_2^2, \{Y_iY_j\}_{(1\leq i < j \leq n)}, \{Y_1^2-Y_i^2\}_{(2 \leq i \leq n)}).
    \end{equation}
   
    Let $\mathbb{X}$ be a minimal apolar set of $F.$   Since $O(n)$ fixes $F$ by Lemma \ref{Lem:NonZeroCoordinatesAfterRotation}, we can assume that for $P_i = [u_i:v_i:a_{i,1}: \ldots :a_{i,n}] \in \mathbb{X}$, $a_{i,1} \neq 0$.
    \noindent Hence $g = Y_1 \in (T/I({\mathbb{X}}))_1$ is a non-zerodivisor in $T/I({\mathbb{X}})$. Let $V = \langle Y_2, Y_3, \ldots, Y_n\rangle_{\mathbb{C}} \subset (T/I({\mathbb{X}}))_1$. Note that $gV \subset (F^{\perp}/I({\mathbb{X}}))$, and $\dim V = n-1$. By Lemma \ref{Lem:UpperboundHF}, we get $$\HF(\mathbb{X}, 2) \leq 4n - (n-1) = 3n +1.$$
    Thus
    \begin{equation}\label{EQ:IX2forX1X2}
        \dim  I({\mathbb{X}})_2\geq \binom{n+3}{2} - 3n-1.     
    \end{equation}
    
    \noindent Let $\mathbb{Y}$ be the projection of $\mathbb{X}$ on $V(X_1,X_2)$ from the line $V(Y_1,\ldots,Y_n)$. 
    Therefore by Lemma \ref{Lem:ProjectionOverSetOfVariablesSaturated} we have 
        \begin{equation}
            I(\mathbb{Y}) = (X_1,X_2)+(I({\mathbb{X}}) \cap \, \mathbb{C}[Y_1, \ldots, Y_n])
        \end{equation}
    Since $I(\mathbb{X}) \subset F^{\perp}$, $I(\mathbb{X})_1 = 0.$ Hence $I(\mathbb{Y})_1 = (X_1,X_2)_1$. Therefore 
    \begin{equation}
        \HF(\mathbb{Y}, 1)= n.
    \end{equation}
    We have $I(\mathbb{Y})_2 = (X_1, X_2)_2 + (I({\mathbb{X}}) \cap \, \mathbb{C}[Y_1, \ldots, Y_n])_2.$ 
    We now use Grassmann formula to estimate
    \[
    \dim I({\mathbb{X}})_2 \cap \, \mathbb{C}[Y_1, \ldots, Y_n].
    \]
    By Equation \eqref{EQ:PerpF2} we note that $\dim (F^{\perp})_2 =  \binom{n}{2}+(n-1) +2$ and $$\dim (F^{\perp})_2 \cap \, \mathbb{C}[Y_1, \ldots, Y_n] = \binom{n}{2}+(n-1).$$ Thus we get that
    
        \begin{equation} \label{Equation:dimI(Y)}
        \dim I({\mathbb{X}})_2 \cap \, \mathbb{C}[Y_1, \ldots, Y_n]\geq\dim I({\mathbb{X}})_2+\dim (F^{\perp})_2 \cap \, \mathbb{C}[Y_1, \ldots, Y_n]-\dim (F^\perp)_2=\binom{n}{2},
        \end{equation}
    and also
        \begin{equation}\label{EQ:HFat2}
            \HF(\mathbb{Y}, 2) \leq \binom{n+1}{2} - \binom{n}{2} = n.
        \end{equation}
     Since $I(\mathbb{Y})$ is an ideal of points and  $\HF(\mathbb{Y} ,1) = n $, {by \cite{GMR83} we have
    \[
    \Delta \HF(\mathbb{Y}) = (1, \HF(\mathbb{Y}, 1) -  \HF(\mathbb{Y}, 0), \ldots,  \HF(\mathbb{Y}, i) -  \HF(\mathbb{Y}, i-1), \ldots )
    \]
    is an O-sequence. Hence we obtain $\HF(\mathbb{Y},2)= n.$ Moreover, by Macaulay's criterion (see \cite[Theorem 4.2.10]{BH98}) for O-sequences, $\Delta \HF(\mathbb{Y}, 1) = 0$ implies $\Delta \HF(\mathbb{Y}, l) = 0$ for all $l \geq 1.$} Thus, we obtain   
    \[\HF(\mathbb{Y}, 2)= n.\]
    Hence the Hilbert function of $\mathbb{Y}$ is given by
            \[
            \begin{array}{r|c|c|c|c}
                i & 0 & 1 & 2 & \cdots \\
                \hline
                \HF(\mathbb{Y},i) & 1 & n & n& n \\
                     
            \end{array}
                \]
    Therefore, $\mathbb{Y}$ is a set of $n$ points in a general position.

\vskip 5mm

    We can now proceed with the  proof of (ii).  Let $\mathbb{Y} = \{Q_1,\ldots,Q_n\}$. 
    From part (i) we have 
            \begin{equation*}
                \mathbb{X} \subset \bigcup_{i=1}^n \langle \,\ell, Q_i\rangle,
            \end{equation*}
    where $\langle \,\ell, Q_i\rangle$ is the plane passing through $\ell$ and $Q_i.$        
    Let $P_1=[1:0:\cdots:0]=V(X_2,Y_1, \ldots,Y_n),$ $P_2=[0:1:0:\cdots:0]=V(X_1,Y_1, \ldots,Y_n)$, and 
     $\mathcal{H}_i= \langle P_1,P_2,Q_1,\ldots, \hat{Q_i} ,\ldots, Q_n \rangle$ for $1\leq i\leq n$. Since $\mathcal{H}_i$ is a hyperplane, we let $\mathcal{H}_i=V(L_i),$ where $L_i$ is a linear form not involving $X_1$ and $X_2.$
     We {\em claim} that $Q_i\notin \mathcal{H}_i$.
     \begin{proof}[{Proof of the claim}]
         We prove the claim by contradiction. If $Q_i\in \mathcal{H}_i$ then $\mathbb{X} \subseteq \mathcal{H}_i.$ This implies that $I(\mathcal{H}_i)\subseteq I(\mathbb{X})\subseteq F^{\perp}.$  This yields $L_i\in F^{\perp},$ a contradiction.
     \end{proof} Let $L_i=\sum_{j=1}^n\alpha_{ij}y_j$, and consider the points of $\mathbb{X}$ not on $V(L_i)$, that is the points $\mathbb{X}\cap\Lambda_i$. Computing we get 
            \begin{equation*}
                \begin{split}
                     F^{\perp} : I(\mathcal{H}_i)=F^{\perp} : (\alpha_{i1} y_1 +  \cdots + \alpha_{in} y_n) & = ((\alpha_{i1} Y_1 +  \cdots + \alpha_{in} Y_n) \circ F)^{\perp}\\
                        & =(2x_1x_2(\alpha_{i1} y_1 +  \cdots + \alpha_{in} y_n))^\perp
                \end{split}
            \end{equation*}
    where the first equality holds by \cite[Lemma 3.2]{CCCGW18}. After a change of coordinates $G= x_1x_2L_i$ is a monomial and \cite[Proposition 3.1]{CCG12} yields $\rk G = 4$. Note that
    \[
    I(\mathbb{X}\cap\Lambda_i)=I(\mathbb{X}) : I(\mathcal{H}_i)\subseteq F^{\perp} : I(\mathcal{H}_i)=G^\perp.
    \]
    Therefore, each of the $n$ planes $\Lambda_i = \langle\ell,Q_i\rangle$ contains at least $4$ points of $\mathbb{X}$. Note that $V(Y_1,\ldots,Y_n)$ is the only line in the intersection of any two planes and that, by Lemma \ref{Lemma:eComputability&ForbiddenLocus}, $V(Y_1,\ldots,Y_n) \subseteq \mathcal{F}_F.$ Hence we obtain at least $4n$ points in the union of the $n$ planes. Since $\mid\mathbb{X}\mid = 4n$,
    each plane $\Lambda_i$ contains exactly $4$ points. {The proof is now completed.}
    \end{proof}

\begin{remark}\label{HFofX1X2}
    In the proof of Proposition \ref{prop:projfromline}, we proved that Equation \eqref{EQ:HFat2} is an equality. Hence, equality holds in Equations \eqref{Equation:dimI(Y)} and \eqref{EQ:IX2forX1X2}. Thus we proved that $\HF(\mathbb{X}, 2) = 3n+1.$
\end{remark}

We now determine the ideals of minimal apolar sets to $F$. Note that the geometry of minimal apolar sets can still vary quite widely depending on the parameters $\alpha_i$ and $\beta_i$ as shown in Remark \ref{Remark: ProejctiontoW}.
{
\begin{proposition}\label{Prop:structureOfIdealX1X2}
    Let $F=x_1x_2(y_1^2+\cdots+y_n^2)$ with $n \geq 2$. If $\mathbb{X}$ is a minimal apolar set of $F$, then there exists $\sigma \in O(n)$ such that ideal $I(\sigma(\mathbb{X}))$ is given by
    \[
    I(\sigma(\mathbb{X})) = (\{Y_iY_j\}_{1 \leq i <j \leq n}, \,X_1^2+\sum_{i=2}^n{\alpha_i}(Y_1^2-Y_i^2), \,X_2^2+\sum_{i=2}^n{\beta_i}(Y_1^2-Y_i^2) )
    \]
    where $\alpha_i, \beta_i \neq 0$ and $\sum_{i=2}^n \alpha_i\neq 0,\, \sum_{i=2}^n \beta_i \neq 0.$
\end{proposition}}
\begin{proof}
    Let $\mathbb{Y}$ be the projection of $\mathbb{X}$ on $V(X_1,X_2)$ from the line $V(Y_1,\ldots,Y_n)$ By the proof of Theorem \ref{prop:projfromline} (ii) we have $\mathbb{Y} = \{Q_1, \ldots, Q_n\}$ is a set of $n$ points in general position. Moreover, by Lemma \ref{Lemma:ProjectionisApolar} we have $\mathbb{Y}$ is apolar to $G=(y_1^2+ \cdots +y_n^2).$ 
    {Thus by \cite[Proposition 3.1]{ForbiddenLocus} we have $Q_i \notin V(\sum_{i=j}^nY_j^2)$ for $1 \leq i \leq n.$ Since the action of $O(n)$ on $\mathbb{P}^{n-1}\setminus V(\sum_{j=1}^nY_j^2)$ has exactly one orbit, there exists $\sigma \in O(n)$ such that  $\sigma(\mathbb{Y}) = \{E_1, \ldots, E_n\}$ where $E_i$ is the standard $i^{th}$-coordinate point for $1 \leq i \leq n$.}
    Hence $I(\sigma(\mathbb{Y}) )\supseteq (Y_iY_j \mid 1 \leq i < j \leq n).$ Hence $(Y_iY_j \mid 1 \leq i < j \leq n) \subseteq I(\sigma(\mathbb{X}))_2.$
    By Remark \ref{HFofX1X2} we have $\dim I(\sigma(\mathbb{X}))_2 = \binom{n}{2}+2$. Hence there exist quadratic forms $\mathcal{Q}_1,\mathcal{Q}_2 \in I(\mathbb{X})_2$ such that
    \[I(\sigma(\mathbb{X}))_2 = (\{Y_iY_j\}_{1 \leq i<j \leq n}, \, \mathcal{Q}_1,\mathcal{Q}_2).\]
    Since $\mathcal{Q}_1,\mathcal{Q}_2 \in F^{\perp},$ and $\HF(\sigma(\mathbb{Y}), 2) =\binom{n}{2}$, $\mathcal{Q}_1,\mathcal{Q}_2$ has at least one of the coefficient of $X_1^2, X_2^2$ non-zero. 
    Thus
    \[\mathcal{Q}_1= X_1^2+\sum_{i=2}^n \alpha_i (Y_1^2-Y_i^2),\, \, \mathcal{Q}_2 = X_2^2+\sum_{i=2}^n \beta_i (Y_1^2-Y_i^2). \]
    And we {\em claim} that $\alpha_i\neq 0, \beta_i\neq 0$ for $2 \leq i \leq n$ and $ \sum_{i=2}^n \alpha_i, \sum_{i=2}^n \beta_i \neq 0.$ 
    \begin{proof}[{Proof of the claim}]
    By contradiction assume $\alpha_i=0$ for some $i$ and consider the plane defined by $\mathcal{P}_i = (Y_1, \ldots, \hat Y_i, \ldots, Y_n).$ Note that, by Proposition \ref{prop:projfromline} (ii), $\overline{I(\sigma(\mathbb{X}))} \subset T/\mathcal{P}_i$ is an ideal of 4 points. Moreover, if $\alpha_i =0$, then  $X_1 \in \overline{I(\sigma(\mathbb{X}))}$. Hence, the Hilbert function of $\HF(T/\overline{I(\sigma(\mathbb{X}))},2) \leq 2.$
    This is a contradiction. Similar arguments can be repeated to complete the proof.
    \end{proof}
    
    \noindent Hence, we write
    \[I(\sigma(\mathbb{X})) = \bigcap_{i =1}^n \, (Y_1, Y_2, \ldots,\hat{Y_i}, \ldots, Y_n, \mathcal{Q}_1,\mathcal{Q}_2\,)+J,\]
    where $J$ is generated in degree at least three in $I(\mathbb{X}).$
    
    Consider \[
        \bar {\mathcal{Q}_1}, \bar{\mathcal{Q}_2}\in\frac{\mathbb{C}[X_1,X_2,Y_1,\ldots,Y_n]}{(Y_1,\ldots,\hat Y_i,\ldots,Y_{n})}\simeq{\mathbb{C}[X_1,X_2,Y_{i}]},
        \]
    and note that $\bar{\mathcal{Q}_1} = X_1^2-k_1Y_i^2, \, \bar{\mathcal{Q}_2} = X_2^2-k_2Y_i^2$ with $k_1,k_2\neq 0$. In particular, $(\bar{\mathcal{Q}_1},\bar{\mathcal{Q_2}})$ is an ideal of four points. By Proposition \ref{prop:projfromline}, we know that $(\bar{\mathcal{Q}_1},\bar{\mathcal{Q_2}})+\bar J$ is the ideal of the same four points. Thus, $(\bar{\mathcal{Q}_1},\bar{\mathcal{Q_2}}) \supseteq \bar J$  and hence  
    $$(Y_1, Y_2, \ldots,\hat{Y_i}, \ldots, Y_n, \mathcal{Q}_1,\mathcal{Q}_2\,)\supseteq J$$
    for all $i$. The proof is now completed.
\end{proof}

{
We now improve part (ii) of Proposition \ref{prop:projfromline}.
\begin{proposition}
    Using the notation of Proposition \ref{prop:projfromline},  let $\mathbb{Y}=\{Q_1,\ldots,Q_n\}$ and let $\ell=V(Y_1,\ldots,Y_n)$. If we consider the plane $\Lambda_i=\langle \ell, Q_i \rangle$, where $\langle \ell, Q_i \rangle$ denotes the linear span of $\ell$ and $Q_i,$ then  for $1\leq i\leq n$,  a set of four distinct points in $\mathbb{X}\cap\Lambda_i$ is  contained in two lines passing through $Q_i$.
\end{proposition}
\begin{proof}
    We want to show that the four points in each of the planes $\Lambda_i$ lie on a pair of lines passing through the point $Q_i$. After a change a coordinates we may assume $Q_i$ to be the coordinate point $E_i$, and we may assume the ideal of the four points to be 
    $$(Y_1, Y_2, \ldots,\hat{Y_i}, \ldots, Y_n, \mathcal{Q}_1,\mathcal{Q}_2\,).$$
   Thus, the ideal of the four points two contains degree two forms $\mathcal{Q}_1, \mathcal{Q}_2$ such that
   \[
   \bar{\mathcal{Q}_1}, \bar{\mathcal{Q}_2} \in T/(Y_1, Y_2, \ldots,\hat{Y_i}, \ldots, Y_n)
   \]
   with $\bar{\mathcal{Q}_1} \neq 0, \bar{\mathcal{Q}_2} \neq 0$. This is enough to conclude the proof.
\end{proof}}

We now determine the Hilbert function of any minimal apolar set to $F$.

\begin{proposition}\label{Prop:HFfor2factors}
    If $F= x_1x_2(y_1^2+y_2^2+ \cdots +y_n^2)$, then the Hilbert function of any minimal apolar set $\mathbb{X}$ is 
    \[
    \begin{array}{r|c|c|c|c|c|c}
        i & 0 & 1& 2 & 3& 4 & \cdots \\
        \hline
        \HF(\mathbb{X}, i) & 1 & n+2 & 3n+1  & 4n & 4n &4n
    \end{array}
    \]
    
\end{proposition}

\begin{proof} 
    By Proposition \ref{Prop:structureOfIdealX1X2}, after a change of coordinates, we have $I(\mathbb{X}) = (\{Y_iY_j\}_{1 \leq i <j \leq n}, \,Q_1, \,Q_2)$ where $Q_1 = X_1^2+\sum_{i=2}^n{\alpha_i}(Y_1^2-Y_i^2),$ and $Q_2 =X_2^2+\sum_{i=2}^n{\beta_i}(Y_1^2-Y_i^2)$ . We consider the lexicographic ordering $X_1>X_2>Y_1 > \cdots>Y_n$ on $\mathbb{C}[X_1, X_2,Y_1, \ldots, Y_n]$ and we compute a Gr\"obner basis $G$ of $I(\mathbb{X})$. In particular,  we will show that $G =(\{Y_iY_j\}_{1\leq i<j\leq n}, \, Q_1,\, Q_2)$ is such a basis and this is enough to prove the result.

    To prove that $G$ is Gr\"obner basis we prove that each S-polynomials is zero mod $G$.
     Let $f_{i,j}  = Y_iY_j$ and note that the S-polynomial of two monomials is automatically zero. Thus, we only need to consider the following two cases.
     \begin{equation*}
         \begin{split}
             S(f_{i,j}, Q_1) & = \frac{Y_iY_jX_1^2}{Y_iY_j}(Y_iY_j) -  \frac{Y_iY_jX_1^2}{X_1^2}Q_1 \\
             & =  Y_iY_j\sum_{k=2}^n \alpha_k (Y_1^2-Y_k^2) \\
             & = 0 \mod G.
         \end{split}
     \end{equation*}
     Similarly, $S(f_{i,j}, Q_2) = 0 \mod G.$
     Moreover, 
     \begin{equation*}
         \begin{split}
             S(Q_1,Q_2) & = \frac{X_1^2X_2^2}{X_1^2}Q_1 -\frac{X_1^2X_2^2}{X_2^2}Q_2 \\
             & = X_2^2 (\sum_{i=2}^n \alpha_i (Y_1^2-Y_i^2)) -X_1^2 (\sum_{i=2}^n \beta_i (Y_1^2-Y_i^2)) \\
             & = X_2^2 (\sum_{i=2}^n \alpha_i (Y_1^2-Y_i^2)) -X_1^2 (\sum_{i=2}^n \beta_i (Y_1^2-Y_i^2)) +X_1^2X_2^2 -X_1^2X_2^2 \\
             & = X_2^2 Q_1 -X_1^2Q_2 \\
             & = 0 \mod G.
         \end{split}
     \end{equation*}
     
     Hence $G$ is a Gr\"obner basis of $I(\mathbb{X}).$ 
     Therefore the initial ideal of $I(\mathbb{X})$ is given by
    \[\text{in}_>(I(\mathbb{X})) = (\{Y_iY_j\}_{1\leq i<j\leq n}, \, X_1^2,X_2^2).\]
    {Since $T/I(\mathbb{X})$ and $T/\text{in}_>(I(\mathbb{X}))$ have the same Hilbert function (see \cite[Chapter 9, Proposition 9]{CLO97}), the proof is now completed.}
\end{proof}

We can now characterize which points appear in any minimal apolar sets to $F$.

\begin{theorem}\label{Thm:forbidden2factors}
     If $F= x_1x_2(y_1^2+y_2^2+ \cdots +y_n^2)$, then the forbidden locus of $F$ \[\mathcal{F}_F = V(X_1X_2) \cup V(Y_1^2 + \cdots Y_n^2).\]
\end{theorem}
\begin{proof}
     By Proposition \ref{Prop:structureOfIdealX1X2} we have $I(\mathbb{X}) = (\{Y_iY_j\}_{1 \leq i <j \leq n}, \,X_1^2+\sum_{i=2}^n{\alpha_i}(Y_1^2-Y_i^2), \,X_2^2+\sum_{i=2}^n{\beta_i}(Y_1^2-Y_i^2) )$ up to a change of coordinates.

    Let $P\in\mathcal{W}_F$ and assume that $P\in V(X_1X_2)$. We will prove that this a contradiction.
    If $X_1(P) = 0$ or $X_2(P)=0$, then, by the shape of the ideal, we get that $P=[0:0: \cdots :0]$. Thus, $V(X_1X_2) \subseteq \mathcal{F}_F.$

   Now, let  $Q = [a:b:c_1: \cdots :c_n] \in \mathcal{W}_F$ and  assume by contradiction that $Q\in V\left(\sum_{i=1}^n Y_i^2\right)$. 
   Let $\mathbb{X}$ be a minimal apolar set containing $Q$. Following Proposition \ref{prop:projfromline} we consider $\mathbb{Y}$ the projection of $\mathbb{X}$ from $V(Y_1, \ldots, Y_n)$ to $V(X_1,X_2)$. Since $\mid \mathbb{Y}\mid =n$, we obtain $\mathbb{Y}$ is a minimal apolar set of $G = y_1^2+\cdots +y_n^2$ by Lemma \ref{Lemma:ProjectionisApolar}. Note that, by Lemma \ref{Lemma:eComputability&ForbiddenLocus}, the image of $Q$ under the projection exists and it is $[0:0:c_1: \cdots :c_n] \in \mathcal{W}_G$. This is a contradiction by Proposition \cite[Proposition 3.1]{ForbiddenLocus} since $\mathcal{F}_G=V(\sum_{i=1}^n Y_i^2)$. Hence, $V(Y_1^2+ \cdots +Y_n^2) \subseteq \mathcal{F}_F.$

Thus we proved that $\mathcal{F}_F\supseteq V(X_1X_2) \cup V(Y_1^2 + \cdots Y_n^2)$. Now we prove that equality holds.
{
Let $Q=[a:b:c_1:c_2: \cdots :c_n]$ such that $ab \neq 0$ and  $\sum_{i=1}^n c_i^2 \neq 0$. Consider the ideal 
\[I(\mathbb{X}) = (\{Y_iY_j\}_{1 \leq i < j \leq n}, X_1^2-Y_1^2- \cdots -Y_{n-1}^2+(n-1)Y_n^2, X_1^2-X_2^2) \subseteq F^{\perp}.\]
Since the orthogonal group action on the variables $Y_1, \ldots, Y_n$ with $\sum_{i=1}^nY_i^2 \neq 0$ has only one orbit, there exists $\sigma \in O(n)$ such that $\sigma(\mathbb{X}) = \mathbb{X}'$ with $[1:1: c_1:c_2: \cdots :c_n] \in \mathbb{X}'.$ Moreover, scaling the first two coordinates fixes $F$ up to multiplication of scalars. Thus there exists $\mathbb{X}''$ apolar to $F$ containing $Q.$
Since $F$ is fixed by $O(n)$ and by scaling of the first two variables, the Waring locus and the forbidden locus are stable under these actions. Hence  the result follows.}
\end{proof}

We now consider the projections of minimal apolar sets to $F$ on the line $V(X_1,X_2)$.

\begin{proposition}\label{Prop:prpjfromline}
     If $\mathbb{W}$ is the projection of $\mathbb{X}$ from $V(X_1,X_2)$ on $V(Y_1,\ldots,Y_n)$, then $\mathbb{W}$ is a set of $2k$ points with $k\leq n$.
\end{proposition}
\begin{proof}
    By Lemma \ref{Lem:ProjectionOverSetOfVariablesSaturated}, we have $I(\mathbb{W}) = (Y_1, \ldots, Y_n) + I(\mathbb{X}) \cap \mathbb{C}[X_1, X_2].$ To prove that the maximal number of points is $2n$, it is enough to prove that $ (I(\mathbb{X})_{2n} \cap \mathbb{C}[X_1, X_2] )\neq 0.$
    By Proposition \ref{Prop:structureOfIdealX1X2} we have 
        \begin{equation}\label{Eq:primaryDec}
            \begin{split}
                I(\mathbb{X})   & = (\{Y_iY_j\}_{1 \leq i <j \leq n}, X_1^2+\sum_{i=2}^n\alpha_{1i}(Y_1^2-Y_n^2), X_2^2+\sum_{i=2}^n\alpha_{2i}(Y_1^2-Y_i^2)) \\
                                & = \bigcap_{i=1}^n (Y_1, \ldots, \hat{Y_i}, \ldots, Y_n, X_1^2-k_{1i}Y_i^2, X_2^2-k_{2i}Y_i^2)
            \end{split}
        \end{equation}
    where $k_{ji} = \alpha_{ji}$ for $1 \leq j \leq 2$ and $2 \leq i \leq n$, and $k_{j1} = \sum_{i=1}^n \alpha_{ji}$ for $1 \leq j \leq 2.$ Hence $\prod_{i=1}^n (k_{2i}X_1^2-k_{1i}X_2^2) \in I(\mathbb{X}).$ Thus the first part of the proof is completed.

    Now we prove that only an even number of points can occur in $\mathbb{W}$. Note that formula \ref{Eq:primaryDec} gives a primary decomposition of $I(\mathbb{X})$. In particular, to understand the projection $\mathbb{W}$, we need to study the common roots of the polynomials $fk_{2i} = X_1^2 -k_{1i}X_2^2$ for $1\leq i \leq n$. Any pair of these polynomials either has two common roots or no common root at all. Hence, the result follows.
\end{proof}
The previous Proposition \ref{Prop:prpjfromline} shows that the projection $\mathbb{W}$, unlike the projection $\mathbb{Y}$, is not completely determined. In the remark that follows we give more details on $\mathbb{W}$.
\begin{remark}\label{Remark: ProejctiontoW} Keeping the same notation as in Proposition \ref{prop:projfromline}, we note the following:
      \begin{enumerate}
          \item[(i)] If we choose $\alpha_{1i} = \alpha_{2i}$ for all $2 \leq i \leq n,$ then $X_1^2-X_2^2 \in I(\mathbb{X}).$ Thus, $\mathbb{W}$ contains exactly 2 points.
          \item[(ii)] If $\frac{\alpha_{1i}}{\alpha_{2i}} \neq \frac{\alpha_{1j}}{\alpha_{2j}}$ for any $2 \leq i, j \leq n,\,i \neq j$ such that $\frac{\sum_{i=2}^n \alpha_{1k}}{\sum_{k=2}^n\alpha_{2k}} \neq \frac{\alpha_{1i}}{\alpha_{2i}}$ for any $2 \leq i \leq n$, we obtain exactly $2n$ points in $\mathbb{W}.$ For example, choose $\alpha_{1i} = i-1$ for $2 \leq i \leq n-1$, $\alpha_{1n}=e^n$ and $\alpha_{2i}=1$ for all $2 \leq i \leq n.$ 
          \item[(iii)] For $n=3$ we can not get exactly $4$ points in $\mathbb{W}$. Indeed, note that if $\frac{\alpha_2}{\beta_2} \neq \frac{\alpha_3}{\beta_3}$ implies 
              \[
              \frac{\alpha_2 +\alpha_3}{\beta_2 + \beta_3} \neq \frac{\alpha_2}{\beta_2} \quad \textrm{and} \quad \frac{\alpha_2 +\alpha_3}{\beta_2 + \beta_3} \neq \frac{\alpha_3}{\beta_3}
              \]
              This implies $\mathbb{W}$ has exactly 6 points. Further, if  $\frac{\alpha_2}{\beta_2} = \frac{\alpha_3}{\beta_3}$, then we obtain exactly $2$ points in $\mathbb{W}.$ 
          \item[(iv)] For $n \geq 4,$ we can choose $\frac{\alpha_i}{\beta_i}$ such that $\mathbb{W}$ has $4$ points. i.e. choose     $\frac{\alpha_2}{\beta_2} \neq \frac{\alpha_3}{\beta_3} $, and $\frac{\alpha_3}{\beta_3} = \frac{\alpha_i}{\beta_i}$ for all $i \geq 4$ with $\sum_{i=2}^{n-1} \alpha_i = \sum_{i=2}^{n-1} \beta_i = 0.$
    \end{enumerate}   
        
\end{remark}
  
Theorem \ref{thm:B} is now proved by using Propositions \ref{prop:projfromline},\, \ref{Prop:structureOfIdealX1X2},\, \ref{Prop:HFfor2factors},\, \ref{Thm:forbidden2factors}, and \ref{Prop:prpjfromline}.

\begin{example}
    Let $F=(x_1^2+x_2^2)(y_1^2+y_2^2+y_3^2).$ 
    By Remark \ref{Remark: ProejctiontoW} we have that the projection of any minimal apolar set $\mathbb{X}$ from $V(X_1,X_2)$ to $V(Y_1,Y_2,Y_3)$ is either a set of 2 points or a set of 6 points, while the projection of $\mathbb{X}$ from $V(Y_1,Y_2,Y_3)$ to $V(X_1,X_2)$ is always a set of 3 points.
    Hence, by Theorem \ref{thm:B} we obtain that any minimal apolar set $\mathbb{X}$ to $F$ has one of the following shapes described in Figure \ref{fig:A} and Figure \ref{fig:B}. Red dots denote the points in $\mathbb{X}$ in both the pictures.

\begin{figure}[H]
\centering

\begin{minipage}{0.47\textwidth}
\centering
\begingroup
\small
\def\svgwidth{\linewidth}
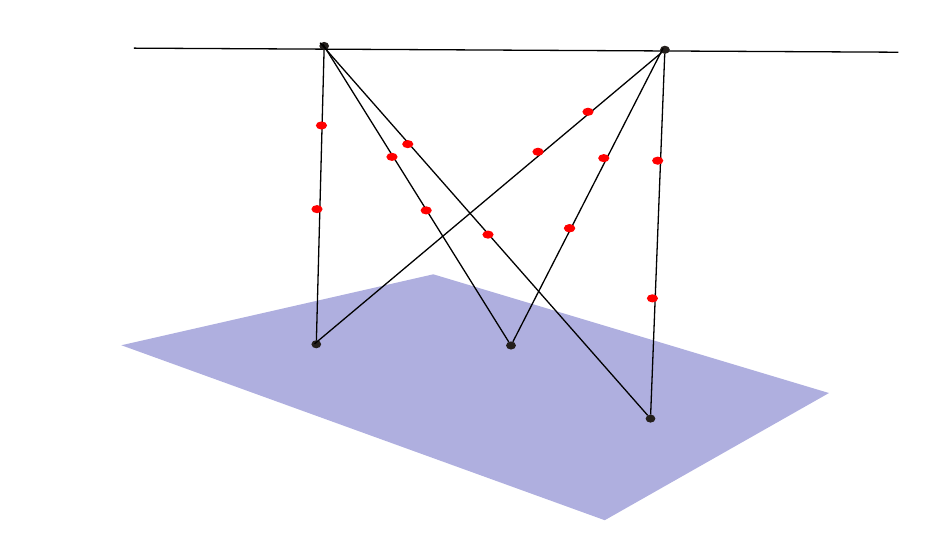
\endgroup
    \caption{Minimal apolar set whose projection from $V(X_1,X_2)$ to $V(Y_1,Y_2,Y_3)$ having 2 points.}
\label{fig:A}
\end{minipage}
\hfill
\begin{minipage}{0.5\textwidth}
\centering
\begingroup
\small
\def\svgwidth{\linewidth}
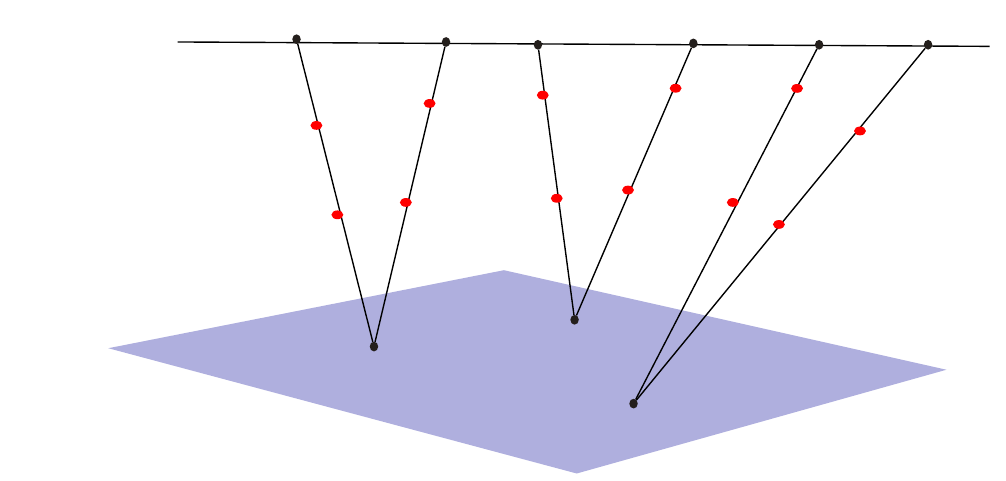
\endgroup
\caption{Minimal apolar set whose projection from $V(X_1,X_2)$ to $V(Y_1,Y_2,Y_3)$ having 6 points.}
\label{fig:B}
\end{minipage}

\end{figure}
\end{example}

\section{The forms  $(x_1^2+x_2^2+ \cdots+x_m^2)(y_1^2+y_2^2+ \cdots+y_n^2)$}
 
In this section, we provide a new lower and a upper bound for the rank of the forms $(x_1^2+x_2^2+ \cdots+x_m^2)(y_1^2+y_2^2+ \cdots+y_n^2)$. 
\begin{proposition}\label{Prop:RKm=3}
    Let $n \geq 2$. If we let $F=(x_1^2+x_2^2+x_3^2)(y_1^2+ \cdots+y_n^2)$, then $\rk F \leq 6n.$
\end{proposition}
\begin{proof}

    Consider the ideal \[I(\mathbb{X}) =(\{Y_iY_j\}_{1 \leq i<j\leq n}, X_1X_2, X_1X_3, X_2^2-X_3^2, X_1^2-X_2^2+Y_1^2+ \cdots Y_{n-1}^2-(n-1)Y_n^2 ). \]
    Note that $I(\mathbb{X}) \subset F^{\perp}$ by Lemma \ref{Lem:PerpF(x)G(y)} and that 
   {\small \[I(\mathbb{X}) = \left(\bigcap_{i=1}^n(Y_1,\ldots, \hat Y_i, \ldots Y_n, X_1, X_2^2-X_3^2, -X_2^2+Y_i^2) \right) \cap \left(\bigcap_{i=1}^n(Y_1,\ldots, \hat Y_i, \ldots Y_n, X_2, X_3, X_1^2+Y_i^2) \right). \]}
   Thus,  $\mathbb{X}$ is an ideal of $6n$ points apolar to $F$.
   Hence, by the Apolarity Lemma, the result is proved.
\end{proof}

\begin{proposition}\label{prop:generalUpperBound}
    Let $m,n \geq 2$. If we let $F=(x_1^2+\cdots+x_m^2)(y_1^2+ \cdots+y_n^2)$, then $\rk F \leq 2mn.$
\end{proposition}
 \begin{proof} Let $q =y_1^2+ \cdots+y_n^2.$ We divide the proof in two cases: $m$ even and $m$ odd.

 \begin{description}
     \item[Case (i)] $m =2k$ for some $k \in \mathbb{N}.$ Then we write
     \begin{equation*}
         \begin{split}
            (x_1^2+\cdots+x_m^2)q & = (x_1^2+x_2^2)q +\cdots+(x_{2k-1}^2+x_{2k}^2)q.
         \end{split}
     \end{equation*}
     Therefore by Proposition \ref{Prop:RankOfForms}, we get
     \[\rk F \leq 4nk = 2n \cdot 2k = 2mn.\]
     \item[Case (ii)] $m =2k+1$ for some $k \in \mathbb{N}.$
     Then we write
     \begin{equation*}
         \begin{split}
            (x_1^2+\cdots+x_m^2)q & = (x_1^2+ \cdots+x_{2k-2}^2)q+(x_{2k-1}^2+x_{2k}^2+x_{2k+1}^2)q.
         \end{split}
     \end{equation*}
 \end{description}
    
     Therefore, by case (i) and  Proposition \ref{Prop:RKm=3}, we get
     \[\rk F \leq 2(2k-2)n +6n = 2n \cdot (2k+1) = 2mn. \]
     The proof is completed.
 \end{proof}

\begin{theorem}\label{Thm:RankOfProductofSquares}
    Let $n \geq m \geq 3$. If we let $F=(x_1^2+\cdots+x_m^2)(y_1^2+ \cdots+y_n^2)$,  then \[n(m+2) < \rk F \leq 2mn.\] 
\end{theorem}
 \begin{proof}
     Note that the upper bound comes from Proposition \ref{prop:generalUpperBound}, thus we only need to prove the lower bound.
     Using Lemma \ref{Lem:PerpF(x)G(y)} we get
  {\small \[ F^{\perp} = (\{X_iX_j\}_{1\leq i<j \leq m}, \, \{X_1^2 -X_i^2\}_{2 \leq i \leq m},\, \{Y_iY_j\}_{1\leq i < j \leq n}, \, \{Y_1^2-Y_i^2\}_{2 \leq i \leq n}).\]}
    Set $I = (Y_1, \ldots,Y_n)$, $t= \alpha_1Y_1 + \cdots + \alpha_n Y_n$, and $J = (F^{\perp} : I) + (t)$. Thus we get
 $$
        J = (\{X_iX_j\}_{1\leq i<j \leq m}, \, \{X_1^2 -X_i^2\}_{2 \leq i \leq m}, \{Y_iY_j\}_{1\leq i < j \leq n},Y_i^2\}_{1 \leq i \leq n}) +(t).
$$
Computing we get
\[
\begin{array}{r|c|c|c|c|c}
    i & 0 & 1 & 2 & 3 &4  \\ \hline
    \HF(T/J, i) & 1 & m+n-1  & m(n-1) & n-1 & 0
\end{array}
\]

\vspace{.5 em}
and thus $\sum_{i \geq 0} \HF(T/J, i) = n(m+2).$
Hence, by \cite[Theorem 3.3]{CCCGW18}, we get  \[\rk F \geq \sum_{i \geq 0} \HF(T/J, i) =n(m+2).\]

We now prove by contradiction that $\rk F >n(m+2).$  We assume that $\rk F = n(m+2)$ and we let $\mathbb{X}$ being a minimal apolar subset of $F$. Note that $F$ is $e$-computable using $I= (Y_1, \ldots, Y_n)$ and $t = \sum_{i=1}^n 
\alpha_i Y_i,$. Thus Lemma  \ref{Lemma:eComputability&ForbiddenLocus} yields $V(Y_1, Y_2, \ldots, Y_n) \subseteq \mathcal{F}_F.$

\noindent Let $V=F^{\perp}_2 \cap \mathbb{C}[Y_1,\ldots, Y_n]_2$ and note that $\dim V = \binom{n}{2} + n-1.$ If we let $\dim V/I(\mathbb{X}) = \dim V - \alpha,$ then $\dim I(\mathbb{X})_2 \cap \mathbb{C}[Y_1,\ldots, Y_n]_2 = \alpha.$ 
Thus, by Lemma \ref{Lem:UpperboundHF}, we have 
\[\HF(\mathbb{X}, 2) \leq n(m+2) - \dim V/I(\mathbb{X}).\]
Therefore 
\[\dim I(\mathbb{X})_2 \geq \binom{n+m+1}{2} -n(m+2) + \dim V/I(\mathbb{X}). \]
Applying Grassmann's formula, we get 
{\small
\begin{equation}\label{EQ:alphaGEQ1}
    \begin{split}
        \alpha=\dim I(\mathbb{X})_2 \cap \mathbb{C}[Y_1,\ldots, Y_n]_2 & \geq \dim I(\mathbb{X})_2 + \dim V - \dim (V + I(\mathbb{X})_2) \\
                 & \geq \dim I(\mathbb{X})_2 + \dim V - \dim F^{\perp}_2 \\
           & \geq \binom{n+m+1}{2} -n(m+2) + \dim V/I(\mathbb{X})  +\dim V - \dim F^{\perp}_2 \\
         & \geq \binom{n}{2}.
    \end{split}
\end{equation}}
and thus $\alpha\geq \binom{n}{2}$.

We now show that $\alpha= \binom{n}{2}$. Note that $(X_1, X_2, \ldots, X_m) + I(\mathbb{X}) \cap \mathbb{C}[Y_1, \ldots, Y_n]$ is the ideal of the projection of $\mathbb{X}$ onto $V((X_1, X_2, \ldots, X_m)$; we call this projection $\mathbb{Y}$.
Since $\HF(\mathbb{Y}, 1) = n$, we get $\HF(\mathbb{Y}, 2) \geq n$.
Thus
$$\HF(\mathbb{Y}, 2) = \binom{n+1}{2} - \dim I(\mathbb{X})_2 \cap \mathbb{C}[Y_1, \ldots, Y_n]_2 \geq n$$
and
\begin{equation}\label{EQ:alphaGEQ2}
    \alpha  \leq \binom{n+1}{2}  - n = \binom{n}{2}.
\end{equation}
Therefore by Equation \eqref{EQ:alphaGEQ1}, \eqref{EQ:alphaGEQ2}, we have 
\[\dim I(\mathbb{X})_2 \cap \mathbb{C}[Y_1, \ldots, Y_n]_2 = \binom{n}{2}.\]
Therefore, $\HF(\mathbb{Y}, 2) = n$ and thus $I(\mathbb{Y})$ is an ideal of $n$ points. Therefore the apolar set $\mathbb{X}$ is contained in $n$ $m$-dimensional linear spaces each passing through  $V(Y_1,\ldots,Y_n)$.

    \noindent Moreover,
    \begin{equation*}
        \begin{split}
            F^{\perp} : (\alpha_1 Y_1 +  \cdots + \alpha_n Y_n) & = ((\alpha_1 Y_1 +  \cdots + \alpha_n Y_n) \circ F)^{\perp}\\
                & =\left((x_1^2+x_2^2+ \cdots x_m^2)(\alpha_1 y_1 +  \cdots + \alpha_n y_n)\right)^\perp
        \end{split}
    \end{equation*}
    where the first equality holds by \cite[Lemma 3.2]{CCCGW18}. Note that $G = (x_1^2+x_2^2+ \cdots +x_m^2)(\alpha_1 y_1 +  \cdots + \alpha_n y_n)$ is such that $\rk G = 2m.$ Since $V(Y_1, Y_2, \ldots, Y_n) \subseteq \mathcal{F}_F,$ $\mathbb{X}$ contains $2mn$ distinct points. Hence a contradiction since $2mn>n(m+2)$.
 \end{proof}

{
\begin{remark}
    Our conjecture is that the form $F = (x_1^2+\cdots+x_m^2)(y_1^2+ \cdots+y_n^2)$ for $m,n \geq 2$ has the Waring rank $\rk F = 2mn.$ For $m=2$ and any $n,$ the Waring rank is $4n.$ If the conjecture turns out to be true, then by using Lemma \ref{Lem:UpperboundHF} and using  similar ideas as in the proofs of Theorem \ref{thm:A} and Theorem \ref{thm:B}, we can get the ideal of minimal apolar sets and their Hilbert functions.
\end{remark}
}

\section{Cactus rank, border rank, and VSP}
In this section, we collect results about the border rank and the cactus rank of our forms, for example see \cite{BBM14, BK13} for some general facts about these ranks. 
We recall that, given $F \in S_d$, the cactus rank of $F$ is
\[
\Cr F = \min\{{\text{length}(\mathbb{X}) }: \mathbb{X} \subset \mathbb{P}(T_1) \text{ with } \dim \mathbb{X} = 0, I(\mathbb{X}) \subseteq F^{\perp} \},
\]
and the border rank $\brk F$ is the smallest $r$ such that $F$ is the limit of polynomials of Waring rank $r.$
We also investigate the variety of sum of powers (see \cite{RS00}) of a given form $F$ with respect to some non-negative integer $r$, namely
\[
VSP(F,r)=\lbrace \mathbb{X}\in\mathcal{H}{ilb}_r(\mathbb{P}^n) : I(\mathbb{X})\subseteq F^\perp\rbrace,
\]
where $\mathcal{H}{ilb}_r(\mathbb{P}^n)$ denotes the Hilbert scheme of length $r$ zero-dimensional subschemes of $\mathbb{P}^n$.
We begin investigating the forms $x^2(y_1^2+\ldots+y_n^2)$.

\begin{lemma}\label{Lem:CactusBorderRankm=1}
    If we let $F=x^2(y_1^2+\cdots+y_n^2)$ for $n\geq 2$, then $\Cr F= \brk F=n+2.$
\end{lemma}
\begin{proof}
    By Lemma \ref{Lem:PerpF(x)G(y)} we have $\HF(T/F^{\perp}, 2) = n+2.$ Hence $\Cr  F \geq n+2$ and $\brk F \geq n+2.$ We now obtain an upper bound for the cactus rank of $F.$ 
    Consider the ideal \[I = (\{Y_iY_j\}_{1\leq i<j \leq n}, \{Y_1^2-Y_i^2\}_{2 \leq i \leq n}).\]
    By Lemma \ref{Lem:PerpF(x)G(y)}  we have $I\subseteq F^{\perp}$. 
    
    We now show that $I$ is a saturated ideal. Let $g \in T$ be such that there exists $l \in \mathbb{N}$ for which $X^lg \in I$. Thus $\bar{X^l}\bar{g} = 0 \in T/I.$
    Since $\bar{X}$ is a non-zerodivisor in $T/I$ we get $\bar{g} = 0.$ Thus, $g \in I.$ Therefore, $I$ is a saturated ideal.
    The ideal $I$ defines a smoothable scheme by \cite[Proposition 2.10]{CN09}. Therefore,
    $I$ defines a zero-dimensional smoothable scheme of length $n+2$.
    Since the smoothable rank is an upper bound on both the cactus rank and the border rank (see \cite[Section 2.1]{BB15} or \cite[Lemma 5.17]{IK99}),  we get that $\Cr F \leq n+2$, $\brk F \leq n+2$. Hence, the equality holds. 
\end{proof}

\begin{proposition}\label{Prop:VSP(F,3n)}
    If we let $F=x^2(y_1^2+\cdots+y_n^2)$ for $n \geq 2$, then $\dim VSP(F,3n)=2n-1+\binom{n}{2}$.    
\end{proposition}
\begin{proof}
    We define the following incidence correspondence
\[
VSP(F,3n)\times O(n)\supseteq\Sigma=\lbrace (\mathbb{X},\sigma) : \pi(\sigma(\mathbb{X}))=\lbrace E_1,\ldots,E_n\rbrace \rbrace
\]
where $\pi$ is the projection from $\mathbb{P}^n$ onto the hyperplane $V(X)$ and the points $E_i$ are the coordinate points in it.

Consider the projection map
\[
\Psi: \Sigma \longrightarrow VSP(F,3n).
\]
Note that, by Proposition \ref{Prop:ShapeofIdealMonomialproductSumofSquares} $\Psi$ is a surjective map having fibers isomorphic to the symmetric group over $n$ elements. Thus, $\dim \Sigma=\dim VSP(F,3n)$.

We \emph{claim} that there exists a rational map
\[
\Phi: \Sigma \DashedArrow[->,densely dashed    ] \mathbb{A}^{2n-1}
\]
such that $\Phi$ is dominant and its fibers are isomorphic to $O(n)$. Hence, $\dim \Sigma=2n-1+\binom{n}{2}$ and the desired result follows.
\begin{proof}[{ Proof of the claim.}] Given a general element $\mathbb{X}\in VSP(F,3n)$, by Proposition \ref{Prop:ShapeofIdealMonomialproductSumofSquares} and using its notation, there exists a $\sigma\in O(n)$ such that $I(\sigma(\mathbb{X}))$ is completely determined by the linear forms $L_i$. Using the degree two generators of $I(\sigma(\mathbb{X}))$ we can assume the linear forms to be
\[
L_i=a_iX+b_{i1}Y_1+b_{ii}Y_i
\]
for $2\leq i\leq n$. Working out the computations, a degree three generator of $I(\sigma(\mathbb{X}))$ can be assumed to be
\[
X^3+\left(\sum_{i=2}^n a_i\right)XY_1^2-\sum_{i=2}^n a_i XY_i^2+\left(\sum_{i=2}^n b_{i1}\right)Y_1^3+\sum_{i=2}^n b_{ii}Y_i^3
\]
and thus we define $\Phi(\mathbb{X},\sigma)=(a_2,\ldots,a_n,\sum_{i=2}^n b_{i1},b_{22},\ldots,b_{nn})$.

To study the map $\Phi$ we use Theorem A as follows. We consider $\mathbb{X}'=\sigma(\mathbb{X})$ which we know is supported on the lines $l_i=\langle P,E_i\rangle$ and we note that $\mathbb{X}'\cap l_i= V(F_i(\alpha))\cap l_i$ for the degree three polynomials
\[
F_1(\alpha)= X^3+ \sum_{i=2}^n a_iXY_1^2 + \sum_{i=2}^n b_{i1}Y_1^3 \mbox{ and } F_i(\alpha)=X^3-a_iXY_i^2+b_{ii}Y_i^3 \mbox{ for } i\geq 2,
\]
where $\alpha=\Phi(\mathbb{X},\sigma)=(a_2,\ldots,a_n,\sum_{i=2}^n b_{i1},b_{22},\ldots,b_{nn})$.
Since $\mathbb{X}'$ is a reduced set of $3n$ points each polynomial $F_i$ avoids the relevant discriminant locus. Thus, $\Phi$ is a dominant map.

Given $\alpha\in\mbox{Im}\, \Phi$ we have that
\[
\Phi^{-1}(\alpha)=\lbrace (\sigma(\mathbb{X}_\alpha),\sigma^{-1}): \sigma \in O(n) \rbrace
\]
where $\mathbb{X}_\alpha$ is the reduced set of $3n$ points $\bigcup_{i= 1}^n V(F_i(\alpha))\cap l_i$. Hence, $\Phi^{-1}(\alpha)$ is isomorphic to $O(n)$.

\end{proof}
The proof of the proposition is completed.
\end{proof}
We now investigate the form $F=(x_1^2+ x_2^2)(y_1^2+\cdots+y_n^2)$ for $n\geq 2$.  In the rest of the section, we present the cactus rank and the border rank of $F$, and compute the dimension of $VSP(F, 4n).$
\begin{lemma}
    If we let $F=(x_1^2+ x_2^2)(y_1^2+\cdots+y_n^2)$ for $n\geq 2$, then \begin{align*}
        \Cr F=2n+4 {\text{ and  } \, 2n+2 \leq \brk F \leq 2n+4.}
    \end{align*} 
\end{lemma}
\begin{proof}
    Since \begin{align*}
        F & = x_1^2(y_1^2+\cdots+y_n^2)+  x_2^2(y_1^2+\cdots+y_n^2),
    \end{align*}
    by Lemma \ref{Lem:CactusBorderRankm=1} we obtain 
    \begin{align}\label{EQ:CactusUB}
        \Cr F \leq 2n+4, \, \text{and}\, \brk F \leq 2n+4.
    \end{align}
    By Lemma \ref{Lem:PerpF(x)G(y)} we have that $F^{\perp}$ is generated in degree 2, and length $\ell(T/F^{\perp}) = 4n+8.$  Hence by \cite[Corollary 1]{RS11} we obtain 
    \begin{align}\label{EQ:CactusLB}
        \Cr F \geq \frac{1}{2}(4n+8) = 2n+4.
    \end{align}
    Hence $\Cr F =2n+4.$
Moreover, Lemma \ref{Lem:PerpF(x)G(y)} also yields that  $\HF(T/F^{\perp}, 2 ) = 2n+2.$ Hence a lower bound on the border rank is proved, and this completes the proof of the lemma. 
\end{proof}

\begin{proposition}
    If we let $F=(x_1^2+x_2^2)(y_1^2+\cdots+y_n^2)$ for $n \geq 2$, then $\dim VSP(F,4n)= 2n-2+\binom{n}{2}$.    
\end{proposition}
\begin{proof} This proof is similar to the proof of Proposition \ref{Prop:VSP(F,3n)}.
    We consider the  incidence correspondence
\[
VSP(F,4n)\times O(n)\supseteq\Sigma=\lbrace (\mathbb{X},\sigma) : \pi(\sigma(\mathbb{X}))=\lbrace E_1,\ldots,E_n\rbrace \rbrace
\]
where $\pi$ is the projection $\mathbb{P}^{n+1} \to V(X_1, X_2)$.

We consider the map
\[
\Psi: \Sigma \longrightarrow VSP(F,4n).
\]
which is surjective with $0$-dimensional fibers. Thus, $\dim \Sigma=\dim VSP(F,4n)$.

We \emph{claim} that there exists a rational map
\[
\Phi: \Sigma \DashedArrow[->,densely dashed    ] \mathbb{A}^{2n-2}
\]
such that $\Phi$ is dominant and its fibers are isomorphic to $O(n)$. Hence, $\dim \Sigma= 2n-2+\binom{n}{2}$ and the desired result follows.
\begin{proof}[{ Proof of the claim.}]
Given an element $\mathbb{X}\in VSP(F,4n)$, by Proposition \ref{Prop:structureOfIdealX1X2} there exists a $\sigma\in O(n)$ such that $I(\sigma(\mathbb{X}))$ is completely determined by $\alpha_i, \beta_i$ for $2 \leq i \leq n$. Thus we define
\[
\Phi(\mathbb{X}, \sigma) = (\alpha_2, \ldots, \alpha_n,\beta_2, \ldots, \beta_n).
\]
We consider $\mathbb{X}'=\sigma(\mathbb{X})$ which we know is supported on the plane $l_i=\langle \ell,E_i\rangle$ where $\ell = V(Y_1, \ldots, Y_n)$ and we note that $\mathbb{X}'\cap l_i$ is a set of 4 points described by two degree 2 polynomials $\mathcal{Q}_1 = X_1^2 - \sum_{i=2}^n \alpha_i (Y_1^2-Y_i^2)$ and  $\mathcal{Q}_2 = X_2^2 - \sum_{i=2}^n \beta_i (Y_1^2-Y_i^2).$ Since $\mathbb{X}'$ is a reduced set of $4n$ points, we obtain $\alpha_i \neq 0,$ $\beta_i \neq 0$ for all $2 \leq i \leq n$ and $\sum_{i=2}^n \alpha_i \neq 0,$ $\sum_{i=2}^n \beta_i \neq 0.$ Hence $\Phi$ is a dominant map.

Given $\alpha\in\mbox{Im} \,\Phi$ we have that
\[
\Phi^{-1}(\alpha)=\lbrace (\sigma(\mathbb{X}_\alpha),\sigma^{-1}): \sigma \in O(n) \rbrace
\]
where $\mathbb{X}_\alpha$ is the reduced set of $4n$ points $\bigcup_{i=1}^n V(\mathcal{Q}_1, \mathcal{Q}_2)\cap l_i$. Hence, $\Phi^{-1}(\alpha)$ is isomorphic to $O(n)$.

\end{proof}
The proof is completed.
\end{proof}
\bibliographystyle{abbrv}
\bibliography{main}

@article {ForbiddenLocus,
    AUTHOR = {Carlini, Enrico and Catalisano, Maria Virginia and Oneto,
              Alessandro},
     TITLE = {Waring loci and the {S}trassen conjecture},
   JOURNAL = {Adv. Math.},
  FJOURNAL = {Advances in Mathematics},
    VOLUME = {314},
      YEAR = {2017},
     PAGES = {630--662},
      ISSN = {0001-8708,1090-2082},
   MRCLASS = {14N05 (13F20 13P05 14N20)},
  MRNUMBER = {3658727},
MRREVIEWER = {Fyodor\ L.\ Zak},
       DOI = {10.1016/j.aim.2017.05.008},
       URL = {https://doi.org/10.1016/j.aim.2017.05.008},
}

@article {GMR83,
    AUTHOR = {Geramita, A. V. and Maroscia, P. and Roberts, L. G.},
     TITLE = {The {H}ilbert function of a reduced {$k$}-algebra},
   JOURNAL = {J. London Math. Soc. (2)},
  FJOURNAL = {Journal of the London Mathematical Society. Second Series},
    VOLUME = {28},
      YEAR = {1983},
    NUMBER = {3},
     PAGES = {443--452},
      ISSN = {0024-6107,1469-7750},
   MRCLASS = {13H15 (14M05)},
  MRNUMBER = {724713},
MRREVIEWER = {D.\ Kirby},
       DOI = {10.1112/jlms/s2-28.3.443},
       URL = {https://doi.org/10.1112/jlms/s2-28.3.443},
}

@article {CCCGW18,
    AUTHOR = {Carlini, Enrico and Catalisano, Maria Virginia and Chiantini,
              Luca and Geramita, Anthony V. and Woo, Youngho},
     TITLE = {Symmetric tensors: rank, {S}trassen's conjecture and
              {$e$}-computability},
   JOURNAL = {Ann. Sc. Norm. Super. Pisa Cl. Sci. (5)},
  FJOURNAL = {Annali della Scuola Normale Superiore di Pisa. Classe di
              Scienze. Serie V},
    VOLUME = {18},
      YEAR = {2018},
    NUMBER = {1},
     PAGES = {363--390},
      ISSN = {0391-173X,2036-2145},
   MRCLASS = {14Q20 (14M10 14N05 15A69)},
  MRNUMBER = {3783793},
MRREVIEWER = {Jaydeep\ V.\ Chipalkatti},
}

@article{CCG12,
title = {The solution to the {W}aring problem for monomials and the sum of coprime monomials},
journal = {Journal of Algebra},
volume = {370},
pages = {5-14},
year = {2012},
issn = {0021-8693},
doi = {https://doi.org/10.1016/j.jalgebra.2012.07.028},
url = {https://www.sciencedirect.com/science/article/pii/S0021869312003730},
author = {Enrico Carlini and Maria Virginia Catalisano and Anthony V. Geramita}
}

@book{IK99,
    author = {Anthony Iarrobino and Vassil Kanev},
    title = {Power Sums, {G}orenstein Algebras, and Determinantal Loci},
    seriesTitle = {Lecture Notes in Mathematics},
    publisher = {Springer Berlin, Heidelberg},
    year = {1999},
}

@article{CVG16,
title = {Real and complex Waring rank of reducible cubic forms},
journal = {Journal of Pure and Applied Algebra},
volume = {220},
number = {11},
pages = {3692-3701},
year = {2016},
issn = {0022-4049},
doi = {https://doi.org/10.1016/j.jpaa.2016.05.007},
url = {https://www.sciencedirect.com/science/article/pii/S0022404916300305},
author = {Enrico Carlini and Emanuele Ventura and Cheng Guo}
}

@article {AH95,
    AUTHOR = {Alexander, J. and Hirschowitz, A.},
     TITLE = {Polynomial interpolation in several variables},
   JOURNAL = {J. Algebraic Geom.},
  FJOURNAL = {Journal of Algebraic Geometry},
    VOLUME = {4},
      YEAR = {1995},
    NUMBER = {2},
     PAGES = {201--222},
      ISSN = {1056-3911,1534-7486},
   MRCLASS = {14N10 (14F17 14Q15)},
  MRNUMBER = {1311347},
MRREVIEWER = {Fyodor\ L.\ Zak},
}

@book {Lan12,
    AUTHOR = {Landsberg, J. M.},
     TITLE = {Tensors: geometry and applications},
    SERIES = {Graduate Studies in Mathematics},
    VOLUME = {128},
 PUBLISHER = {American Mathematical Society, Providence, RI},
      YEAR = {2012},
     PAGES = {xx+439},
      ISBN = {978-0-8218-6907-9},
   MRCLASS = {15-01 (14N05 15A69 20G05)},
  MRNUMBER = {2865915},
MRREVIEWER = {M.\ R.\ Pournaki},
       DOI = {10.1090/gsm/128},
       URL = {https://doi.org/10.1090/gsm/128},
}

@article{BBT13,
title = {Waring decompositions of monomials},
journal = {Journal of Algebra},
volume = {378},
pages = {45-57},
year = {2013},
issn = {0021-8693},
doi = {https://doi.org/10.1016/j.jalgebra.2012.12.011},
url = {https://www.sciencedirect.com/science/article/pii/S0021869312006308},
author = {Weronika Buczyńska and Jarosław Buczyński and Zach Teitler}
}

@article {BM21,
    AUTHOR = {Brustenga i Moncus\'i, Laura and Masuti, Shreedevi K.},
     TITLE = {The {W}aring rank of binary binomial forms},
   JOURNAL = {Pacific J. Math.},
  FJOURNAL = {Pacific Journal of Mathematics},
    VOLUME = {313},
      YEAR = {2021},
    NUMBER = {2},
     PAGES = {327--342},
      ISSN = {0030-8730,1945-5844},
   MRCLASS = {13F20 (11P05 14N07)},
  MRNUMBER = {4323417},
MRREVIEWER = {Nathan\ Grieve},
       DOI = {10.2140/pjm.2021.313.327},
       URL = {https://doi.org/10.2140/pjm.2021.313.327},
}

@article{Lee16,
title = {Power sum decompositions of elementary symmetric polynomials},
journal = {Linear Algebra and its Applications},
volume = {492},
pages = {89-97},
year = {2016},
issn = {0024-3795},
doi = {https://doi.org/10.1016/j.laa.2015.11.018},
url = {https://www.sciencedirect.com/science/article/pii/S0024379515006813},
author = {Hwangrae Lee}
}

@article{GG24,
  title={Symmetric rank of some reducible forms},
  author={Colarte-G{\'o}mez, Liena and Galuppi, Francesco},
  note={Preprint available at arXiv:2410.20390},
  year={2024}
}

@book {Lan14GCT,
    AUTHOR = {Landsberg, J. M.},
     TITLE = {Geometry and complexity theory},
    SERIES = {Cambridge Studies in Advanced Mathematics},
    VOLUME = {169},
 PUBLISHER = {Cambridge University Press, Cambridge},
      YEAR = {2017},
     PAGES = {xi+339},
      ISBN = {978-1-107-19923-1},
   MRCLASS = {14Q20 (20C30 20G05 68Q15 68Q17 68Q25 68W30)},
  MRNUMBER = {3729273},
MRREVIEWER = {Matteo\ Gallet},
       DOI = {10.1017/9781108183192},
       URL = {https://doi.org/10.1017/9781108183192},
}

@article {QuantumInformationTheory,
    AUTHOR = {Bruzda, Wojciech and Friedland, Shmuel and \.Zyczkowski,
              Karol},
     TITLE = {Rank of a tensor and quantum entanglement},
   JOURNAL = {Linear Multilinear Algebra},
  FJOURNAL = {Linear and Multilinear Algebra},
    VOLUME = {72},
      YEAR = {2024},
    NUMBER = {11},
     PAGES = {1796--1859},
      ISSN = {0308-1087,1563-5139},
   MRCLASS = {81P40 (14N07 15A69 65K10 90C27)},
  MRNUMBER = {4769279},
MRREVIEWER = {Hemant\ Kumar\ Mishra},
       DOI = {10.1080/03081087.2023.2211717},
       URL = {https://doi.org/10.1080/03081087.2023.2211717},
}

@article {Rez92,
    AUTHOR = {Reznick, Bruce},
     TITLE = {Sums of even powers of real linear forms},
   JOURNAL = {Mem. Amer. Math. Soc.},
  FJOURNAL = {Memoirs of the American Mathematical Society},
    VOLUME = {96},
      YEAR = {1992},
    NUMBER = {463},
     PAGES = {viii+155},
      ISSN = {0065-9266,1947-6221},
   MRCLASS = {11E76 (11P05 52A21)},
  MRNUMBER = {1096187},
       DOI = {10.1090/memo/0463},
       URL = {https://doi.org/10.1090/memo/0463},
}

@article{Flavi25,
title = {Upper bounds for the rank of powers of quadrics},
journal = {Linear Algebra and its Applications},
volume = {707},
pages = {49-79},
year = {2025},
issn = {0024-3795},
doi = {https://doi.org/10.1016/j.laa.2024.11.009},
url = {https://www.sciencedirect.com/science/article/pii/S0024379524004294},
author = {Cosimo Flavi}
}

@article{Sylvester1851,
  title={Sketch of a memoir on elimination, transformation, and canonical forms},
  author={Sylvester, James Joseph},
  journal={Cambridge and Dublin Mathematical Journal},
  volume={6},
  pages={186--200},
  year={1851}
}

@article{sylvester1851lx,
  title={Lx. on a remarkable discovery in the theory of canonical forms and of hyperdeterminants},
  author={Sylvester, James Joseph},
  journal={The London, Edinburgh, and Dublin Philosophical Magazine and Journal of Science},
  volume={2},
  number={12},
  pages={391--410},
  year={1851},
  publisher={Taylor \& Francis}
}

@inproceedings {DFP20-ApolarAlgebra,
    AUTHOR = {DiPasquale, Michael and Flores, Zachary and Peterson, Chris},
     TITLE = {On the apolar algebra of a product of linear forms},
 BOOKTITLE = {I{SSAC}'20---{P}roceedings of the 45th {I}nternational
              {S}ymposium on {S}ymbolic and {A}lgebraic {C}omputation},
     PAGES = {130--137},
 PUBLISHER = {ACM, New York},
      YEAR = {[2020] \copyright 2020},
      ISBN = {978-1-4503-7100-1},
   MRCLASS = {14Q65 (14N05 14N20 68W30)},
  MRNUMBER = {4144031},
       DOI = {10.1145/3373207.3404014},
       URL = {https://doi.org/10.1145/3373207.3404014},
}

@article {Lan22,
    AUTHOR = {Landsberg, J. M.},
     TITLE = {Secant varieties and the complexity of matrix multiplication},
   JOURNAL = {Rend. Istit. Mat. Univ. Trieste},
  FJOURNAL = {Rendiconti dell'Istituto di Matematica dell'Universit\`a{} di
              Trieste. An International Journal of Mathematics},
    VOLUME = {54},
      YEAR = {2022},
     PAGES = {Art. No. 11, 21},
      ISSN = {0049-4704,2464-8728},
   MRCLASS = {14N07 (14L35 14Q20 15A69 68Q15)},
  MRNUMBER = {4595168},
}

@article {LT10,
    AUTHOR = {Landsberg, J. M. and Teitler, Zach},
     TITLE = {On the ranks and border ranks of symmetric tensors},
   JOURNAL = {Found. Comput. Math.},
  FJOURNAL = {Foundations of Computational Mathematics. The Journal of the
              Society for the Foundations of Computational Mathematics},
    VOLUME = {10},
      YEAR = {2010},
    NUMBER = {3},
     PAGES = {339--366},
      ISSN = {1615-3375,1615-3383},
   MRCLASS = {14N05 (15A03 15A15 68Q25)},
  MRNUMBER = {2628829},
MRREVIEWER = {Enrico\ Carlini},
       DOI = {10.1007/s10208-009-9055-3},
       URL = {https://doi.org/10.1007/s10208-009-9055-3},
}

@article {LO15,
    AUTHOR = {Landsberg, Joseph M. and Ottaviani, Giorgio},
     TITLE = {New lower bounds for the border rank of matrix multiplication},
   JOURNAL = {Theory Comput.},
  FJOURNAL = {Theory of Computing. An Open Access Journal},
    VOLUME = {11},
      YEAR = {2015},
     PAGES = {285--298},
      ISSN = {1557-2862},
   MRCLASS = {68Q17 (15A99)},
  MRNUMBER = {3376667},
MRREVIEWER = {Shengxin\ Zhu},
       DOI = {10.4086/toc.2015.v011a011},
       URL = {https://doi.org/10.4086/toc.2015.v011a011},
}

@article {BC12,
    AUTHOR = {Bernardi, Alessandra and Carusotto, Iacopo},
     TITLE = {Algebraic geometry tools for the study of entanglement: an
              application to spin squeezed states},
   JOURNAL = {J. Phys. A},
  FJOURNAL = {Journal of Physics. A. Mathematical and Theoretical},
    VOLUME = {45},
      YEAR = {2012},
    NUMBER = {10},
     PAGES = {105304, 13},
      ISSN = {1751-8113,1751-8121},
   MRCLASS = {81Q80 (14N05)},
  MRNUMBER = {2897068},
MRREVIEWER = {Fyodor\ L.\ Zak},
       DOI = {10.1088/1751-8113/45/10/105304},
       URL = {https://doi.org/10.1088/1751-8113/45/10/105304},
}

@misc{Shitov16,
      title={How hard is the tensor rank?}, 
      author={Yaroslav Shitov},
      year={2016},
      eprint={1611.01559},
      archivePrefix={arXiv},
      primaryClass={math.CO},
      url={https://arxiv.org/abs/1611.01559}, 
}

@book {AM69,
    AUTHOR = {Atiyah, M. F. and Macdonald, I. G.},
     TITLE = {Introduction to commutative algebra},
 PUBLISHER = {Addison-Wesley Publishing Co., Reading, Mass.-London-Don
              Mills, Ont.},
      YEAR = {1969},
     PAGES = {ix+128},
   MRCLASS = {13.00},
  MRNUMBER = {242802},
MRREVIEWER = {Johnny\ A.\ Johnson},
}

@article{polynomialMatrixMultiplication18,
author = {Chiantini, Luca and Hauenstein, Jonathan D. and Ikenmeyer, Christian and Landsberg, Joseph M. and Ottaviani, Giorgio},
title = {Polynomials and the exponent of matrix multiplication},
journal = {Bulletin of the London Mathematical Society},
volume = {50},
number = {3},
pages = {369-389},
doi = {https://doi.org/10.1112/blms.12147},
url = {https://londmathsoc.onlinelibrary.wiley.com/doi/abs/10.1112/blms.12147},
eprint = {https://londmathsoc.onlinelibrary.wiley.com/doi/pdf/10.1112/blms.12147},
year = {2018}
}

@article{NPhard-TensorRank,
author = {Hillar, Christopher J. and Lim, Lek-Heng},
title = {Most Tensor Problems Are NP-Hard},
year = {2013},
issue_date = {November 2013},
publisher = {Association for Computing Machinery},
address = {New York, NY, USA},
volume = {60},
number = {6},
issn = {0004-5411},
url = {https://doi.org/10.1145/2512329},
doi = {10.1145/2512329},
journal = {J. ACM},
month = nov,
articleno = {45},
numpages = {39}
}

@article{DGIJL26,
author = {Dutta, Pranjal and Gesmundo, Fulvio and Ikenmeyer, Christian and Jindal, Gorav and Lysikov, Vladimir},
title = {Geometric complexity theory for product-plus-power},
year = {2026},
issue_date = {Jan 2026},
publisher = {Academic Press, Inc.},
address = {USA},
volume = {132},
number = {C},
issn = {0747-7171},
url = {https://doi.org/10.1016/j.jsc.2025.102458},
doi = {10.1016/j.jsc.2025.102458},
journal = {J. Symb. Comput.},
month = jan,
numpages = {32}
}

@article{Survey,
author = {Shpilka, Amir and Yehudayoff, Amir},
title = {Arithmetic Circuits: A survey of recent results and open questions},
year = {2010},
issue_date = {March 2010},
publisher = {Now Publishers Inc.},
address = {Hanover, MA, USA},
volume = {5},
number = {3–4},
issn = {1551-305X},
url = {https://doi.org/10.1561/0400000039},
doi = {10.1561/0400000039},
journal = {Found. Trends Theor. Comput. Sci.},
month = mar,
pages = {207–388},
numpages = {182}
}

@article{Kumar20,
author = {Kumar, Mrinal},
title = {On the Power of Border of Depth-3 Arithmetic Circuits},
year = {2020},
issue_date = {March 2020},
publisher = {Association for Computing Machinery},
address = {New York, NY, USA},
volume = {12},
number = {1},
issn = {1942-3454},
url = {https://doi.org/10.1145/3371506},
doi = {10.1145/3371506},
journal = {ACM Trans. Comput. Theory},
month = feb,
articleno = {5},
numpages = {8}
}

@book {BH98,
    AUTHOR = {Bruns, Winfried and Herzog, J\"urgen},
     TITLE = {Cohen-{M}acaulay rings},
    SERIES = {Cambridge Studies in Advanced Mathematics},
    VOLUME = {39},
 PUBLISHER = {Cambridge University Press, Cambridge},
      YEAR = {1993},
     PAGES = {xii+403},
      ISBN = {0-521-41068-1},
   MRCLASS = {13H10 (13-02)},
  MRNUMBER = {1251956},
MRREVIEWER = {Matthew\ Miller},
}

@book {CLO97,
    AUTHOR = {Cox, David and Little, John and O'Shea, Donal},
     TITLE = {Ideals, varieties, and algorithms},
    SERIES = {Undergraduate Texts in Mathematics},
   EDITION = {Second},
      NOTE = {An introduction to computational algebraic geometry and
              commutative algebra},
 PUBLISHER = {Springer-Verlag, New York},
      YEAR = {1997},
     PAGES = {xiv+536},
      ISBN = {0-387-94680-2},
   MRCLASS = {13P10 (13-01 14-01 14Qxx 68Q40)},
  MRNUMBER = {1417938},
}

@article {BK13,
    AUTHOR = {Bernardi, Alessandra and Ranestad, Kristian},
     TITLE = {On the cactus rank of cubics forms},
   JOURNAL = {J. Symbolic Comput.},
  FJOURNAL = {Journal of Symbolic Computation},
    VOLUME = {50},
      YEAR = {2013},
     PAGES = {291--297},
      ISSN = {0747-7171,1095-855X},
   MRCLASS = {13F20},
  MRNUMBER = {2996880},
MRREVIEWER = {Enrico\ Carlini},
       DOI = {10.1016/j.jsc.2012.08.001},
       URL = {https://doi.org/10.1016/j.jsc.2012.08.001},
}

@article{BBM14,
title = {A comparison of different notions of ranks of symmetric tensors},
journal = {Linear Algebra and its Applications},
volume = {460},
pages = {205-230},
year = {2014},
issn = {0024-3795},
doi = {https://doi.org/10.1016/j.laa.2014.07.036},
url = {https://www.sciencedirect.com/science/article/pii/S002437951400487X},
author = {Alessandra Bernardi and Jérôme Brachat and Bernard Mourrain}
}

@article {RS00,
    AUTHOR = {Ranestad, Kristian and Schreyer, Frank-Olaf},
     TITLE = {Varieties of sums of powers},
   JOURNAL = {J. Reine Angew. Math.},
  FJOURNAL = {Journal f\"ur die Reine und Angewandte Mathematik. [Crelle's
              Journal]},
    VOLUME = {525},
      YEAR = {2000},
     PAGES = {147--181},
      ISSN = {0075-4102,1435-5345},
   MRCLASS = {14C05 (14J45 14N15)},
  MRNUMBER = {1780430},
MRREVIEWER = {A.\ A.\ Iarrobino, Jr.},
       DOI = {10.1515/crll.2000.064},
       URL = {https://doi.org/10.1515/crll.2000.064},
}

@article{RS11,
title = {On the rank of a symmetric form},
journal = {Journal of Algebra},
volume = {346},
number = {1},
pages = {340-342},
year = {2011},
issn = {0021-8693},
doi = {https://doi.org/10.1016/j.jalgebra.2011.07.032},
url = {https://www.sciencedirect.com/science/article/pii/S0021869311004728},
author = {Kristian Ranestad and Frank-Olaf Schreyer}
}

@incollection {CGO14,
    AUTHOR = {Carlini, Enrico and Grieve, Nathan and Oeding, Luke},
     TITLE = {Four lectures on secant varieties},
 BOOKTITLE = {Connections between algebra, combinatorics, and geometry},
    SERIES = {Springer Proc. Math. Stat.},
    VOLUME = {76},
     PAGES = {101--146},
 PUBLISHER = {Springer, New York},
      YEAR = {2014},
      ISBN = {978-1-4939-0625-3; 978-1-4939-0626-0},
   MRCLASS = {14N05},
  MRNUMBER = {3213518},
MRREVIEWER = {Zach\ Teitler},
       DOI = {10.1007/978-1-4939-0626-0\_2},
       URL = {https://doi.org/10.1007/978-1-4939-0626-0_2},
}

@article {CN09,
    AUTHOR = {Casnati, Gianfranco and Notari, Roberto},
     TITLE = {On the {G}orenstein locus of some punctual {H}ilbert schemes},
   JOURNAL = {J. Pure Appl. Algebra},
  FJOURNAL = {Journal of Pure and Applied Algebra},
    VOLUME = {213},
      YEAR = {2009},
    NUMBER = {11},
     PAGES = {2055--2074},
      ISSN = {0022-4049,1873-1376},
   MRCLASS = {14C05 (13H10 14M05)},
  MRNUMBER = {2533305},
MRREVIEWER = {Laurent\ Evain},
       DOI = {10.1016/j.jpaa.2009.03.002},
       URL = {https://doi.org/10.1016/j.jpaa.2009.03.002},
}

@article {BB15,
    AUTHOR = {Buczy\'nska, Weronika and Buczy\'nski, Jaros\l aw},
     TITLE = {On differences between the border rank and the smoothable rank
              of a polynomial},
   JOURNAL = {Glasg. Math. J.},
  FJOURNAL = {Glasgow Mathematical Journal},
    VOLUME = {57},
      YEAR = {2015},
    NUMBER = {2},
     PAGES = {401--413},
      ISSN = {0017-0895,1469-509X},
   MRCLASS = {14N05},
  MRNUMBER = {3333949},
MRREVIEWER = {Roberto\ Mu\~noz},
       DOI = {10.1017/S0017089514000378},
       URL = {https://doi.org/10.1017/S0017089514000378},
}

\end{document}